\documentclass[a4paper,12pt]{amsart}
\usepackage{etex}
\usepackage{tikz, tikz-cd}

\usepackage{bm}
\usepackage{amssymb}
\usepackage{mathrsfs}
\usepackage{amsmath, amssymb,amsthm,latexsym,amscd,mathrsfs}
\usepackage{indentfirst}
\usepackage{stmaryrd}
\usepackage{graphicx}
\usepackage{subfigure}
\usepackage{extarrows}
\usepackage{amssymb}
\usepackage{amsmath,amssymb,amsthm,latexsym,amscd,mathrsfs}
\usepackage{indentfirst}
\usepackage{multirow}
\usepackage{latexsym}
\usepackage{amsfonts}
\usepackage{color}
\usepackage{pictexwd,dcpic}
\usepackage{graphicx}
\usepackage{psfrag}
\usepackage{hyperref}
\usepackage{comment}
\usepackage{cases}

\makeatletter
\@namedef{subjclassname@2010}{{\rm2020} Mathematics Subject Classification}
\makeatother

 \newcommand{\R}{\mathbb{R}}

 \newcommand{\ROM}[1]{\mathrm{\uppercase\expandafter{\romannumeral#1}}}
  \theoremstyle{definition}
   \numberwithin{equation}{section} \theoremstyle{plain}

 \newtheorem{thm}{Theorem}
 \newtheorem{lem}{Lemma}[section]

 \newtheorem{rem}{Remark}[section]
 \newtheorem{prop}{Proposition}

 \newtheorem{assert}{Assertion}[subsection]

\newtheorem{ack}{Acknowledgements}   
  \makeatletter

  \numberwithin{equation}{section}
  \allowdisplaybreaks
\makeatother
\title[Clifford systems, harmonic maps and non-negative curvature]{\textbf{Clifford systems, harmonic maps and metrics with non-negative curvature}}
\author[Chao Qian]{Chao Qian}\address{School of Mathematics and Statistics, Beijing Institute of Technology, Beijing
100081, P.R. China}
\email{6120150035@bit.edu.cn}
\author[Z. Z. Tang]{Zizhou Tang}\address{Chern Institute of Mathematics $\&$ LPMC, Nankai University, Tianjin 300071, P. R. China}
\email{zztang@nankai.edu.cn}
\author[W. J. Yan]{Wenjiao Yan}\address{School of Mathematical Sciences, Laboratory of Mathematics and Complex Systems, Beijing Normal University, Beijing, 100875, P. R. China}
\email{wjyan@bnu.edu.cn}
\thanks {The project is partially supported by the NSFC (11871282, 11931007), BNSF (Z190003), Nankai Zhide Foundation.
}


\begin{document}

\maketitle

\begin{abstract}
Associated with a symmetric Clifford system $\{P_0, P_1,\cdots, P_{m}\}$ on $\mathbb{R}^{2l}$, there is a canonical vector bundle
$\eta$ over $S^{l-1}$. For $m=4$ and $8$, we construct explicitly its characteristic map, and determine completely when the sphere bundle
$S(\eta)$ associated to $\eta$ admits a cross-section. These generalize the results in \cite{St51} and \cite{Ja58}. As an application, we establish new harmonic representatives of certain elements in homotopy groups of spheres (cf. \cite{PT97} \cite{PT98}). By a suitable choice of Clifford system, we construct a metric of non-negative curvature on $S(\eta)$ which is diffeomorphic to the inhomogeneous focal submanifold $M_+$ of OT-FKM type isoparametric hypersurfaces with $m=3$.
\end{abstract}

\section{Introduction}\label{intro}

As it is well known, the classification of isoparametric hypersurfaces in unit spheres was accomplished recently.
For $g=4$, i.e., with $4$ distinct constant principal curvatures, the most complicated and beautiful case,
 isoparametric hypersurfaces must be of OT-FKM type except for two homogeneous cases (cf. \cite{CCJ07}, \cite{Imm08},  \cite{Chi13}, \cite{Chi20}).
 Given a symmetric Clifford system $\{P_0, P_1,\cdots, P_{m}\}$ on $\mathbb{R}^{2l}$,
 \cite{FKM81} constructed a Cartan-M\"{u}nzner polynomial $F$ on
$\mathbb{R}^{2l}$:
\begin{equation*}\label{FKM isop. poly.}
F(x) = |x|^4 - 2\displaystyle\sum_{i = 0}^{m}{\langle
P_{i}x,x\rangle^2},
\end{equation*}
 which produces an \emph{OT-FKM type} isoparametric family with multiplicity pair $(m_1,m_2)=(m, l-m-1)$  provided $l-m-1>0$, where $l=k\delta(m)$, $\delta(m)$ is the dimension of the irreducible module of the Clifford algebra $\mathcal{C}_{m-1}$.
According to \cite{FKM81}, when $m\not \equiv 0~ (\mathrm{mod}~4)$, there exists exactly one kind of OT-FKM type.
When $m\equiv 0~(\mathrm{mod} ~4)$,
the family with $P_0P_1\cdots P_m=\pm I_{2l}$ is called \emph{definite}, the others \emph{indefinite}. There are exactly $[\frac{k}{2}]$ non-congruent indefinite families.

Up to an algebraic equivalence, one can always express
\begin{equation*}\label{FKM}
P_0=\left(
\begin{matrix}
I_l & 0 \\
0 & -I_l \\
\end{matrix}\right),\;
P_1=\left(
\begin{matrix}
0 & I_l\\
I_l & 0 \\
\end{matrix}\right), P_{1+\alpha}=\left(
\begin{matrix}
0 & E_{\alpha}\\
-E_{\alpha} & 0 \\
\end{matrix}\right),
~\mathrm{for}~ 1 \leq \alpha\leq m-1,
\end{equation*}
where $\{E_1, E_2,\cdots, E_{m-1}\}$ are skew-symmetric orthogonal transformations on $\mathbb{R}^l$, satisfying
$E_{\alpha}E_{\beta}+E_{\beta}E_{\alpha}=-2\delta_{\alpha\beta}I_l$\,\,
for $1 \leq \alpha,\beta\leq m-1$. It is known that the focal submanifold $M_+$ is quadratic and is given by
\begin{equation*}
M_+
= \left\{(z, w) \in \mathbb{R}^l\oplus\mathbb{R}^l~\vline~\begin{array}{ll}\langle z, w \rangle=0, ~~|z|^2=|w|^2=\frac{1}{2},\\
 \langle z, E_{\alpha}w \rangle=0, ~\alpha=1,\cdots, m-1\end{array}\right\}.\label{M+}
\end{equation*}
Let $\eta$ be the subbundle of $TS^{l-1}$ such that the fiber of $\eta$ at $z\in S^{l-1}$ is the orthogonal complement in $\mathbb{R}^l$ of the $m$-plane
spanned by $\{z, E_1z,\cdots, E_{m-1}z\}$. Then $M_+$ is clearly diffeomorphic to $S(\eta)$, the associated sphere bundle of $\eta$ (cf. \cite{Wa88}, \cite{QT16}).

When $m=1$, $2$ or $4$ in the definite case, $M_+$ is just the real Stiefel manifold $V_2(\mathbb{R}^k)$, complex Stiefel manifold $V_2(\mathbb{C}^k)$, or quaternionic Stiefel manifold $V_2(\mathbb{H}^k)$ respectively, and the characteristic map of the corresponding vector bundle $\eta$ has been given in Theorem 23.4, Theorem 24.3 and Remark 24.11
in the classic book
\cite{St51} by Steenrod. 
When $m=8$ in the definite case, Proposition \ref{1.2} in the present paper reveals that $M_+$ is diffeomorphic to the octonionic Stiefel manifold $V_2(\mathbb{O}^k)$.
As the first main result of this paper, we exhibit the characteristic map of the vector bundle $\eta$ for $m=4$ in the indefinite case and for all cases with $m=8$.
The following theorem is an extension of the results in Steenrod's book \cite{St51}.

\begin{thm}\label{characteristic maps}
Let $\{P_0, P_1,\cdots, P_{m}\}$ be a symmetric Clifford system on $\mathbb{R}^{2l}$ and $l=k\delta(m)$. Then
\begin{itemize}
\item[(i).] For $m=4$ and $\mathrm{Tr}(P_0P_1\cdots P_4)=-8(2p-k+2)$ with $0\leq p\leq k-1$, the characteristic map $\chi: S^{4k-2} \rightarrow \mathrm{SO}(4k-4)$ of $\eta$ is given by
$$\chi(z)(X)=X-2\Big(\langle X, W\rangle_p(1+z_k)^{-2}\Big)\ast_p W,$$
where $z=(z_1,..., z_k)\in S^{4k-2}$, the equator in $S^{4k-1}\subset \mathbb{H}^k$ with $ \mathrm{Re}(z_k)=0$, $W=(z_1,..., z_{k-1})$, and $X\in \mathbb{H}^{k-1}\cong \mathbb{R}^{4k-4}.$ Moreover, $\langle \cdot, \cdot\rangle_p$ and $\ast_p$-operation are defined by (\ref{q1}) and (\ref{q2}), respectively.
\vspace{1mm}
\item[(ii).] For $m=8$ and $\mathrm{Tr}(P_0P_1\cdots P_8)=-16(2p-k+2)$ with $0\leq p\leq k-1$, the characteristic map $\chi: S^{8k-2} \rightarrow \mathrm{SO}(8k-8)$ of $\eta$ is given by
$$\chi(z)(X)=X-2\Big(\langle X, W\rangle_p(1+z_k)^{-2}\Big)\ast_p W,$$
where $z=(z_1,..., z_k)\in S^{8k-2}$, the equator in $S^{8k-1}\subset \mathbb{O}^k$ with $ \mathrm{Re}(z_k)=0$, $W=(z_1,..., z_{k-1})$, and $X\in \mathbb{O}^{k-1}\cong \mathbb{R}^{8k-8}.$ Moreover, $\langle \cdot, \cdot\rangle_p$ and $\ast_p$-operation are defined by (\ref{o1}) and (\ref{o2}), respectively.
\end{itemize}
\end{thm}



For $m=4$, 
composing the characteristic map $\chi: S^{4k-2} \rightarrow \mathrm{SO}(4k-4)$ of $\eta$ with the natural projection $\pi: \mathrm{SO}(4k-4)\rightarrow S^{4k-5}$, we obtain a map $\pi\circ\chi: S^{4k-2}\rightarrow S^{4k-5}$. It is also an interesting problem to determine its homotopy class in $\pi_{4k-2}S^{4k-5}$, which will be studied in Section 2. As a consequence, we show

\begin{prop}\label{quaternion}
Let $m=4$ and $\mathrm{Tr}(P_0P_1P_2P_3P_4)=-8(2p-k+2)$ with $0\leq p\leq k-1$.
The sphere bundle $S^{4k-5}\hookrightarrow M_+^{8k-6}\rightarrow S^{4k-1}$ associated with the vector bundle $\eta$ admits a cross-section if and only if $k-2-2p$ can be divided by $24$.
\end{prop}

Furthermore, the similar case $m=8$ is considered by the following

\begin{prop}\label{octonion}
Let $m=8$ and $\mathrm{Tr}(P_0P_1\cdots P_8)=-16(2p-k+2)$ with $0\leq p\leq k-1$. Then the sphere bundle $S^{8k-9}\hookrightarrow M_+^{16k-10}\rightarrow S^{8k-1}$ associated with the vector bundle $\eta$ admits a cross-section if and only if $k-2-2p$ can be divided by $240$.

In particular, 
the sphere bundle associated with the octonionic Stiefel manifold $V_2(\mathbb{O}^k)$: $S^{8k-9}\hookrightarrow V_2(\mathbb{O}^k)\rightarrow V_1(\mathbb{O}^k)$ admits a cross-section if and only if $k$ can be divided by $240$.
\end{prop}

\begin{rem}\rm{
Two propositions above are extensions of the existence theory for cross-section of James \cite{Ja58}.
In fact, for $m=4$ in the definite case, the existence of the cross-section of $V_2(\mathbb{H}^k)$ was originally obtained in \cite{Ja58}. Moreover, James \cite{Ja58} claimed that a cross-section of $V_2(\mathbb{O}^k)\rightarrow V_1(\mathbb{O}^k)$ occurs when $k=240$. Proposition \ref{octonion} generalizes his result.}
\end{rem}

Recall that a map $f: M \rightarrow N$ between compact Riemannian manifolds is called \emph{harmonic} if it is a critical point of the
energy functional $E(f)=\frac{1}{2}\int_M|df|^2$. An important problem posed by Eells
and Sampson in \cite{ES64} is that: when does a homotopy class of a map between
compact Riemannian manifolds admit a harmonic representative? Particularly, it is a fundamental problem asked by Yau \cite{Yau82} to prove the existence of harmonic representatives for all elements of homotopy groups of spheres. We refer to \cite{Smi75, PT97, PT98} for progresses on this problem.
As our second main result, we give more harmonic representatives for certain elements of homotopy groups of spheres.

\begin{thm}\label{thm1.2}
\begin{itemize}
\item[(i).] The homotopy group $\pi_{14}S^7$ has a harmonic generator.
\item[(ii).] Assume $k\in \mathbb{Z}$ and $k>2$. Then the element $(2j-k+1)(\mathrm{mod}~240)$ in the homotopy group $\pi_{8k-1}S^{8k-8}=\mathbb{Z}_{240}$ has a harmonic  representative for every $j=0,\cdots, k-1$.
\end{itemize}
\end{thm}

In the final part, we focus on the existence of metrics with non-negative sectional curvature on isoparametric hypersurfaces and focal submanifolds in unit spheres. The curvature properties of the induced metrics 
have been investigated in \cite{QTY21}, and the sectional curvatures
are not non-negative in general. It is natural to propose the following

\vspace{2mm}
\noindent \textbf{Problem.}\emph{
Does each isoparametric hypersurface or focal submanifold in the unit sphere admit a metric with non-negative sectional curvature?}
\vspace{2mm}

It can be trivially seen that the homogeneous isoparametric hypersurfaces and the associated focal submanifolds admit metrics with non-negative sectional curvature induced from the bi-invariant metrics on the associated compact Lie groups. As a result, we only need to consider the isoparametric hypersurfaces and focal submanifolds of inhomogeneous OT-FKM type. Making use of Proposition 1 of \cite{Wa88}, we prove in Section \ref{sec4} the following
\begin{prop}\label{M_-}
Each focal submanifold $M_-$ of OT-FKM type admits a metric with non-negative sectional curvature.
\end{prop}
We are left to consider inhomogeneous focal submanifold $M_+$.
\begin{thm}\label{M_+}
The focal submanifold $M_+$ of OT-FKM type with $(m_1, m_2)=(3, 4k-4)$ admits a metric with non-negative
sectional curvature.
\end{thm}

\begin{rem}\rm{
In this case, the corresponding isoparametric foliation is inhomogeneous, as discovered in \cite{OT75}. Moreover, for $k\geq 3$, $M_+$ is not a homogeneous space(cf. \cite{Kra02}).}
\end{rem}

\begin{rem}\rm{
For $( m_1, m_2)=(8, 7)$ and the definite case, the corresponding isoparametric foliation is inhomogeneous. However, $M_+^{22}$ is homogeneously embedded in $S^{31}$ (cf. \cite{FKM81}). More precisely, $M_+^{22}$ is exactly diffeomorphic to $V_2(\mathbb{O}^2)=\{A\in M_{2\times 2}(\mathbb{O})~|~A\bar{A}^t=I_2\}\cong Spin(9)/G_2$, which admits a metric with non-negative sectional curvature.}
\end{rem}

\begin{rem}\rm{
According to \cite{FKM81},
the hypersurface $M$ of OT-FKM type is diffeomorphic to $M_+\times S^m$. It suggests that we only need to consider $M_+$ of OT-FKM type in Problem 1.3. It is still unknown whether each $M_+$ admits a metric with non-negative sectional curvature.}
\end{rem}

The present paper is organized as follows. In Section 2, we explicitly construct the characteristic map of $\eta$ for $m=4$ or $8$. In Section 3, applications related to harmonic maps will be addressed. In Section 4, we investigate the existence of metrics with non-negative curvature on isoparametric families.

\section{Characteristic maps of the vector bundle $\eta$}
In this section, the main purpose is to determine the characteristic maps of $\eta$ for $m=4$ in the indefinite case and for all cases with $m=8$.

\subsection{The $(m_1, m_2)=(4, 4k-5)$ case}
For $m=1, 2,$ or $4$ in the definite case, as we mentioned, the result has been given by Steenrod in \cite{St51}. However, to warm up, we will start with the definite case of $m=4$.

\subsubsection{}\hspace{-2.5mm}\textbf{The definite case of} $(m_1, m_2)=(4, 4k-5)$.\label{def 4}
In this case, the isoparametric hypersurface of OT-FKM type is homogeneous, and the focal submanifold $M_+$ is diffeomorphic to the quaternionic Stiefel manifold $V_2(\mathbb{H}^k)$, which is an $S^{4k-5}$-bundle over $S^{4k-1}$.

For $z=(z_1,\cdots, z_k), w=(w_1,\cdots, w_k)\in \mathbb{H}^{k}$, define the quaternionic inner product by $\langle z, w\rangle_{\mathbb{H}}:=\sum_{i=1}^{k}z_i\overline{w_i}$, and the real inner product by $\langle z, w\rangle:=\mathrm{Re}\langle z, w\rangle_{\mathbb{H}}$. Let $S^{4k-1}$ be the unit sphere in $\mathbb{H}^{k}$ and $S^{4k-2}$ be its equator defined by $\mathrm{Re}(z_k)=0$. Let $N=(0,\cdots,0, 1)\in \mathbb{H}^{k}$.
Then one can construct an embedding
$$\psi_1: (S^{4k-1}\setminus \{-N\})\times S^{4k-5}\rightarrow V_2(\mathbb{H}^{k}), \quad \psi_1(z, X)=(z, Y),$$
where
\begin{eqnarray*}
X\in \mathbb{H}^{k-1}&\cong&\mathbb{H}^{k-1}\oplus 0\subset \mathbb{H}^{k}, ~~W=(z_1,\cdots, z_{k-1}),~~ b=(1+z_k)(1+\overline{z_k})^{-1},\\
&&Y=X\left(
\begin{array}{cc}
I_{k-1}-\overline{W}^t(1+\overline{z_k})^{-1}W,& -\overline{W}^tb\\
\end{array}
\right),
\end{eqnarray*}
and another embedding
$$\psi_2: (S^{4k-1}\setminus \{N\})\times S^{4k-5}\rightarrow V_2(\mathbb{H}^{k}), \quad \psi_2(z, X)=(z, Y),$$
where
\begin{eqnarray*}
X\in \mathbb{H}^{k-1}&\cong&\mathbb{H}^{k-1}\oplus 0\subset \mathbb{H}^{k}, ~~W=(z_1,\cdots, z_{k-1}),~~ b'=(1-z_k)(1-\overline{z_k})^{-1},\\
&&Y=X\left(
\begin{array}{cc}
I_{k-1}-\overline{W}^t(1-\overline{z_k})^{-1}W,& \overline{W}^tb'\\
\end{array}
\right).
\end{eqnarray*}
Then 
we determine in our case,
restricted to $S^{4k-2}$,
$$\psi_1^{-1}(\psi_2(z, X))=(z, XA), ~~\forall~ z\in S^{4k-2},$$
where using $\overline{z_k}=-z_k$, we obtain
\begin{eqnarray*}
A&=&\left(
\begin{array}{cc}
I_{k-1}-\overline{W}^t(1-\overline{z_k})^{-1}W,& \overline{W}^tb' \\
\end{array}
\right)\left(
\begin{array}{c}
I_{k-1}-\overline{W}^t(1+z_k)^{-1}W\vspace{3mm}\\
-\overline{b}W\\
\end{array}
\right)\\
&=&\left(
\begin{array}{c}
I_{k-1}-2\overline{W}^t(1+z_k)^{-2}W\\
\end{array}
\right).
\end{eqnarray*}
In this way, according to the definition in \cite{St51}, the characteristic map is given by
\begin{eqnarray}
\chi:~~ S^{4k-2}&\longrightarrow& \mathrm{Sp}(k-1)\label{10}\\
z&\mapsto& \chi(z)=I_{k-1}-2\overline{W}^t(1+z_k)^{-2}W.\nonumber
\end{eqnarray}

\begin{rem}\rm{
The characteristic map $\chi:S^{4k-2} \rightarrow \mathrm{Sp}(k-1)$ constructed above represents a generator of the homotopy group $\pi_{4k-2}\mathrm{Sp}(k-1)$. This statement should be known. However, for self-containedness, a proof
is given as follows.

Consider the fibration $\mathrm{Sp}(k-1)\hookrightarrow \mathrm{Sp}(k)\overset{p}{\rightarrow} S^{4k-1}$ and the associated
homotopy exact sequence
\begin{equation*}\label{fgh}
\pi_{4k-1}\mathrm{Sp}(k)\overset{p_{\ast}}{\rightarrow} \pi_{4k-1}S^{4k-1} \overset{\Delta}{\rightarrow}\pi_{4k-2}\mathrm{Sp}(k-1) \overset{}{\rightarrow} \pi_{4k-2}\mathrm{Sp}(k)\rightarrow 0.
\end{equation*}
By the Bott periodicity theorem for symplectic groups(cf. \cite{Bo59}), one has
$$\pi_{4k-2}\mathrm{Sp}(k)\cong \pi_{4k-2}\mathbf{Sp}=0.$$
Thus $\Delta$ is surjective, and $\pi_{4k-2}\mathrm{Sp}(k-1)$ is generated by $\Delta [Id_{4k-1}]$.
It follows from 18.4 of \cite{St51}, that $[\chi]=\Delta [Id_{4k-1}]$, and consequently, $[\chi]$ also generates $\pi_{4k-2}\mathrm{Sp}(k-1)$.
Moreover,
$$\pi_{4k-2}\mathrm{Sp}(k-1)\cong \pi_{4k-1}S^{4k-1}/\mathrm{Ker} \Delta\cong \pi_{4k-1}S^{4k-1}/\mathrm{Im} p_{\ast}.$$
By (3) of Theorem 6.13 on P. 216 in \cite{MT91},
\begin{eqnarray*}
&&p_{\ast}\pi_{4k-1}\mathrm{Sp}(k)=\small{\left\{\begin{array}{ll}
(2k-1)!\pi_{4k-1}S^{4k-1}~for~k~odd,\\
2(2k-1)!\pi_{4k-1}S^{4k-1}~for~k~even,
\end{array}\right.}\nonumber
\end{eqnarray*}
where we have made a correction of the statement on P. 216 in \cite{MT91} for the case with even $k$.
Therefore,
\begin{eqnarray*}
&&\pi_{4k-2}\mathrm{Sp}(k-1)=\small{\left\{\begin{array}{ll}
(2k-1)! \;~for~k~odd,\\
2(2k-1)! \;~for~k~even.
\end{array}\right.}\nonumber
\end{eqnarray*}
Here, we have also made corrections for the corresponding statement at P. 218 in \cite{MT91}.}
\end{rem}

\begin{rem}\label{1.3}\rm{
In this case with $m=4, l=k\delta(4)=4k$, combining the characteristic map $\chi:  S^{4k-2}\rightarrow \mathrm{Sp}(k-1)$ with
the natural projection to the first row $\pi: \mathrm{Sp}(k-1)\rightarrow S^{4k-5},$ we obtain
$$\sigma_k:=\pi\circ\chi: ~~S^{4k-2}\rightarrow  S^{4k-5},$$
which is defined by
$$\sigma_k(z_1,\cdots,z_k)=(1-2\overline{z_1}(1+z_k)^{-2}z_1, -2\overline{z_1}(1+z_k)^{-2}z_2,\cdots,-2\overline{z_1}(1+z_k)^{-2}z_{k-1}),$$
where $\mathrm{Re}(z_k)=0$.

When $k=2$, the homotopy class $[\sigma_2]$ of
\begin{eqnarray*}
\sigma_2: ~~S^{6}&\longrightarrow&  S^{3}\\
(z_1, z_2)&\mapsto& \sigma_2(z_1,z_2)=1-2\overline{z_1}(1+z_2)^{-2}z_1
\end{eqnarray*}
generates the homotopy group $\pi_6S^3 = Z_{12}$.

When $k\geq 3$, $[\sigma_k]\in \pi_{4k-2}S^{4k-5}$ which is isomorphic to $\pi_3^S= \mathbb{Z}_{24}$, the third stable homotopy group of spheres.
To investigate the homotopy class $[\sigma_k]$, we recall some basic knowledge on homotopy groups of spheres.
Following \cite{St51}, we define
$$\varrho: S^3\rightarrow \mathrm{SO}(3), \quad x\mapsto \varrho(x),$$
where $\varrho(x)(y)=xy\overline{x}$, $x,y \in S^3\subset \mathbb{H}$ and $\mathrm{Re}(y)=0$. Moreover, define
$$\zeta: S^3\rightarrow \mathrm{SO}(4), \quad x\mapsto \zeta(x),$$
where $\zeta(x)(y)=xy$, $x \in S^3\subset\mathbb{H}$ and $y\in \mathbb{H}$.
To fix the notation, for a continuous function $f: S^m\times S^n \rightarrow S^r$, we define the Hopf construction
\begin{eqnarray*}
H(f): ~~S^{m+n+1}&\longrightarrow& S^{r+1}\\
\big((\cos{t})x, \,(\sin{t})y\big)&\mapsto&(-\cos^2{t}+\sin^2{t}, ~~2\sin{t}\cos{t}f(x, y)),
\end{eqnarray*}
where $(x, y)\in S^m\times S^n, t\in [0, \frac{\pi}{2}]$. In fact, this construction induces a homomorphism $H$, which is actually the $J$-homomorphism
$$H: \pi_{m}\mathrm{SO}(n+1)\rightarrow \pi_{n+m+1}S^{n+1}, \quad [\varpi]\mapsto [H(\varpi)],$$
where
\begin{equation}\label{8}
H(\varpi)\big((\cos{t})x, (\sin{t})y\big)=\left(-\cos^2{t}+\sin^2{t}, 2\sin{t}\cos{t}\varpi(x)(y)\right),
\end{equation}
for $(x, y)\in S^m\times S^n, t\in [0, \frac{\pi}{2}].$

According to \cite{Hu59}, $\pi_7 S^4\cong\mathbb{Z}\oplus \mathbb{Z}_{12}$, where the free part is generated by the homotopy class $[H(\zeta)]$ of the Hopf fibration $H(\zeta): S^7 \rightarrow  S^4$, and the torsion part $\mathbb{Z}_{12}$ is generated by the
homotopy class $[\Sigma H(\varrho)]$ of
$\Sigma H(\varrho)$, the suspension of  $H(\varrho)$.
For $p\geq 1$, $\pi_{7+p}S^{4+p}\cong\mathbb{Z}_{24}$ is generated by $[\Sigma^pH(\zeta)]$. In particular, $[\Sigma ^2 H(\varrho)]=2[ \Sigma H(\zeta)]$ in $\pi _8S^5$(see 23.6 of \cite{St51}).

Following from the proof of Lemma 1 in \cite{Wh46, Wh47},
there exists a map $\psi: S^{4k-2}\rightarrow S^{4k-5}$ defined by
\begin{equation}\label{82}
\psi(z_1, z_2, \cdots, z_k)=\left\{\begin{array}{ll}
(1-2|z_1|^2+2\frac{\overline{z_1}z_kz_1}{|z_1|}, -2\overline{z_1}z_2,\cdots, -2\overline{z_1}z_{k-1}), ~~\,z_1\neq 0,\vspace{2mm}\\
(1-2|z_1|^2, -2\overline{z_1}z_2,\cdots, -2\overline{z_1}z_{k-1}), ~~\,\,z_1=0,
\end{array}\right.
\end{equation}
where $z_1, \cdots.z_{k-1}\in \mathbb{H}$, $z_k\in \mathrm{Im}\mathbb{H}$, such that
$\sigma_k$ is homotopic to $\psi$. In particular, $[H(\varrho)]=-[\sigma_2]\in \pi_{6}S^3$.
On the other hand, choose $\omega: S^3\rightarrow \mathrm{SO}(4k-5)$ by
\begin{equation}\label{83}
\omega(x_1)(x_2,\cdots, x_k)=(\overline{x_1}x_kx_1, -\overline{x_1}x_2,\cdots, -\overline{x_1}x_{k-1}),
\end{equation}
where $x_1\in S^3\subset\mathbb{H}$, $(x_2,\cdots,x_{k-1})\in \mathbb{H}^{k-2}$, $x_k\in \mathrm{Im} \mathbb{H}$.
Clearly, $\psi=H(\omega)$, since if we substitute $(\cos t)x = x_1$ and $(\sin t)y = (x_2, . . . , x_k)$ into (\ref{8}), we obtain (\ref{82}).
Moreover, define $\omega_0: S^3\rightarrow \mathrm{SO}(4k-5)$ by
$$\omega_0(x_1)(x_2,\cdots, x_k)=(\overline{x_1}x_kx_1, x_2,\cdots, x_{k-1}),$$
and for $1\leq i\leq k-2$, define $\omega_i: S^3\rightarrow \mathrm{SO}(4k-5)$ by
$$\omega_i(x_1)(x_2,\cdots, x_k)=(x_2,\cdots, x_i, -\overline{x_1}x_{i+1}, x_{i+2},\cdots, x_{k}).$$
It follows that as homotopy classes in $\pi_{4k-2}S^{4k-5}$,
$$[H(\omega)]=[H(\omega_0)]+[H(\omega_1)]+\cdots+[H(\omega_{k-2})],$$
$$[H(\omega_i)]=-[\Sigma^{4k-9} H(\zeta)],~ i=1,\cdots, k-2,$$
and
$$[H(\omega_0)]=-[\Sigma^{4k-8} H(\varrho)].$$
Therefore, for $k\geq 3$,
$$[\sigma_k]=-(k-2)[\Sigma^{4k-9}H(\zeta)]-[\Sigma^{4k-8} H(\varrho)]=-k[\Sigma^{4k-9}H(\zeta)]\equiv-k~(\mathrm{mod}~24).$$

As is well known, for an $S^p$-bundle over $S^q$ and its characteristic map $\chi: S^{q-1}\rightarrow O(p+1)$, the sphere bundle is trivial if and only if the characteristic map $\chi$ is homotopic to a constant map. Moreover, considering the natural projection to the first row $\pi: O(p+1)\rightarrow S^{p}$, one knows that the sphere bundle admits a cross-section if and only if the map $\pi\circ\chi: S^{q-1}\rightarrow S^p$ is homotopic to a constant map.

Consequently, in this case, the sphere bundle $S^{4k-5}\hookrightarrow M_+^{8k-6}\rightarrow S^{4k-1}$ admits a cross-section if and only if $k$ can be divided by $24$ (see also \cite{Ja58}).}
\end{rem}

In the end of this subsection, for completeness, the corresponding results for the cases of $m=1$ and $2$ are summarized in the following two remarks.

\begin{rem}\rm{
For $m=1$, and $l=k$, combining the characteristic map $\chi:  S^{k-2}\rightarrow \mathrm{O}(k-1)$ of $\eta$ with
the natural projection to the first row $\pi: O(k-1)\rightarrow S^{k-2},$ one gets a quadratic map
$\iota:=\pi\circ\chi: S^{k-2}\rightarrow  S^{k-2}.$
According to Theorem 23.4 in \cite{St51}, $[\iota]\in \pi_{k-2}S^{k-2}\cong\mathbb{Z}$
for $k\geq 3$ and the degree of $\iota$ is $1+(-1)^{k-1}$.}
\end{rem}

\begin{rem}\rm{
For $m=2$ and $l=k\delta(2)=2k$, the homotopy class of the characteristic map $\chi: ~~S^{2k-2}\longrightarrow\mathrm{U}(k-1)$ for $\eta$
generates $\pi_{2k-2}U(k-1)$(cf. \cite{Ker60}). Moreover, according to Bott \cite{Bo58}, $\pi_{2k-2}U(k-1)$ is isomorphic to
$\mathbb{Z}_{(k-1)!}$. Combining the characteristic map $\chi:  S^{2k-2}\rightarrow \mathrm{U}(k-1)$ with
the natural projection to the first row $\pi: U(k-1)\rightarrow S^{2k-3},$ one has
$\varsigma:=\pi\circ\chi: S^{2k-2}\rightarrow  S^{2k-3}$  with
$[\varsigma]\in \pi_{2k-2}S^{2k-3}$. For $k\geq 3$, it is well known that $\pi_{2k-2}S^{2k-3}$ is isomorphic to $\pi_1^S= \mathbb{Z}_2$,
the first stable homotopy group of spheres. Moreover, according to Theorem 24.3 in \cite{St51}, $[\varsigma]=0$ when $k$ is even,
and $[\varsigma]=1$, the generator when $k$ is odd.}
\end{rem}


\subsubsection{}\hspace{-2.5mm}\textbf{The general case of} $(m_1, m_2)=(4, 4k-5)$.
Recall that the quaternionic inner product is defined by $\langle z, w\rangle_{\mathbb{H}}=\sum_{i=1}^{k}z_i\overline{w_i}$ for $z=(z_1,\cdots, z_k), w=(w_1,\cdots, w_k)$ $\in \mathbb{H}^{k}$. Then the real inner product $\langle z, w\rangle$ is equal to $\mathrm{Re}\langle z, w\rangle_{\mathbb{H}}$.
Moreover, for $0\leq p\leq k-1$, define
\begin{equation}\label{q1}
\langle z, w\rangle_p:=\sum_{i=1}^{p}z_i\overline{w}_i+\sum_{j=p+1}^{k-1}w_j\overline{z}_j+z_{k}\overline{w}_k.
\end{equation}

\begin{prop}\label{prop4}
\begin{itemize}
\item[(1).] For $z=(z_1,\cdots, z_k), w=(w_1,\cdots, w_k)\in \mathbb{H}^{k}$, $\langle z, w\rangle_p=0$
if and only if $\mathrm{Re}\langle z, w\rangle_{\mathbb{H}}=0$ and $\sum_{i=1}^{p}z_i\overline{w}_i-\sum_{j=p+1}^{k-1}z_j\overline{w}_j+z_{k}\overline{w}_k\in \mathbb{R}$.\vspace{1mm}

\item[(2).] $N_+:=\{(z, w)\in \mathbb{H}^{k}\oplus\mathbb{H}^{k}~|~|z|=|w|=1, \langle z, w\rangle_p=0\}$
is diffeomorphic to the focal submanifold $M_+$ of OT-FKM type with $(m_1, m_2)=(4, 4k-5)$ and $\mathrm{Tr}(P_0P_1P_2P_3P_4)=-8(2p-k+2)$.
\end{itemize}
\end{prop}

\begin{proof}
Obviously, (1) is true. To prove (2), for $z=(z_1,\cdots, z_k)\in\mathbb{H}^{k}$, we define three orthogonal transformations $E_1, E_2, E_3$ on $\mathbb{H}^{k}$ as follows
\begin{eqnarray}\label{E}
E_1(z)&:=&(\mathrm{i}z_1, \mathrm{i}z_2,\cdots, \mathrm{i}z_p, -\mathrm{i}z_{p+1},\cdots, -\mathrm{i}z_{k-1}, \mathrm{i}z_{k}),\nonumber\\
E_2(z)&:=&(\mathrm{j}z_1, \mathrm{j}z_2,\cdots, \mathrm{j}z_p, -\mathrm{j}z_{p+1},\cdots, -\mathrm{j}z_{k-1}, \mathrm{j}z_{k}),\\
E_3(z)&:=&(\mathrm{k}z_1, \mathrm{k}z_2,\cdots, \mathrm{k}z_p, -\mathrm{k}z_{p+1},\cdots, -\mathrm{k}z_{k-1}, \mathrm{k}z_{k}).\nonumber
\end{eqnarray}
Furthermore, for $(z, w)\in \mathbb{H}^{k}\oplus\mathbb{H}^{k}$, define
\begin{eqnarray}\label{P}
P_0(z, w)&=&(z, -w),\nonumber\\
P_1(z, w)&=&(w, z),\nonumber\\
P_2(z, w)&=&(E_1(w), -E_1(z)),\\
P_3(z, w)&=&(E_2(w), -E_2(z)),\nonumber\\
P_4(z, w)&=&(E_3(w), -E_3(z)).\nonumber
\end{eqnarray}
Consequently, $\{P_0, P_1, P_2, P_3, P_4\}$ is a symmetric Clifford system on $\mathbb{R}^{8k}\cong \mathbb{H}^{2k}$.
A direct computation leads to $\mathrm{Tr}(P_0P_1P_2P_3P_4)=-8(2p-k+2)$. Then following from (1),  the focal submanifold $M_+$ associated
with the given Clifford system is determined by
$M_+=\{(z, w)\in \mathbb{H}^{k}\oplus\mathbb{H}^{k}~|~|z|=|w|=\frac{1}{\sqrt{2}}, \langle z, w\rangle_p=0\}$.
Hence $M_+\cong N_+$.
\end{proof}

For each $\varepsilon\in\mathbb{H}$, $W=(w_1,\cdots, w_{k-1})\in \mathbb{H}^{k-1}$ and $p=0 \cdots, k-1$, define
\begin{equation}\label{q2}
\varepsilon\ast_p W:=(\varepsilon w_1,\cdots, \varepsilon w_p, \overline{\varepsilon} w_{p+1},\cdots, \overline{\varepsilon} w_{k-1}).
\end{equation}
Moreover, for $X=(x_1,\cdots, x_{k-1}),W \in \mathbb{H}^{k-1}=$ $\mathbb{H}^{k-1}\oplus 0\subset \mathbb{H}^{k}$,  recall
$$\langle X, W\rangle_p=\sum_{i=1}^{p}x_i\overline{w}_i+\sum_{j=p+1}^{k-1}w_j\overline{x}_j.$$
Now, we are ready to give a proof for
\vspace{2mm}

\noindent \textbf{Theorem \ref{characteristic maps}.} (i).
\emph{For $m=4$ and $\mathrm{Tr}(P_0P_1\cdots P_4)=-8(2p-k+2)$ with $0\leq p\leq k-1$, the characteristic map $\chi: S^{4k-2} \rightarrow \mathrm{SO}(4k-4)$ of $\eta$ is given by
$$\chi(z)(X)=X-2\Big(\langle X, W\rangle_p(1+z_k)^{-2}\Big)\ast_p W,$$
where $z=(z_1,..., z_k)\in S^{4k-2}$, the equator in $S^{4k-1}\subset \mathbb{H}^k$ with $ \mathrm{Re}(z_k)=0$, $W=(z_1,..., z_{k-1})$, and $X\in \mathbb{H}^{k-1}\cong \mathbb{R}^{4k-4}.$ Moreover, $\langle \cdot, \cdot\rangle_p$ and $\ast_p$-operation are defined by (\ref{q1}) and (\ref{q2}), respectively.}
\vspace{2mm}

\begin{proof}
Let $S^{4k-1}$ be the unit sphere in $\mathbb{H}^{k}$ and $S^{4k-2}$ be its equator defined by $\mathrm{Re}(z_k)=0$. Choose the reference point $N=(0,\cdots,0, 1)\in S^{4k-1}$.

To construct the characteristic map, we first construct an embedding
$$\psi_1: (S^{4k-1}\setminus \{-N\})\times S^{4k-5}\rightarrow N_+, \quad \psi_1(z, X)=(z, Y),$$
where $$X\in \mathbb{H}^{k-1}\cong\mathbb{H}^{k-1}\oplus~ 0\subset \mathbb{H}^{k}, ~~\,\,W=(z_1,\cdots, z_{k-1}),$$
$$b=(1+z_k)(1+\overline{z_k})^{-1}, ~~\,\,\alpha=(1+\overline{z_k})^{-1},$$
and
$$Y=\Big(X-(\langle X, W\rangle_p\alpha)\ast_p W, ~~-\langle X, W\rangle_p b\Big).$$
To prove that $\psi_1$ is well-defined, we need the following lemma.

\begin{lem}\label{2.1}
\begin{itemize}
\item[(1).] For any $\varepsilon\in\mathbb{H}$, $W=(z_1,\cdots, z_{k-1})\in \mathbb{H}^{k-1}$,
$$\langle \varepsilon\ast_p W, W\rangle_p=\varepsilon\langle W, W\rangle_{\mathbb{H}};$$
\item[(2).] For any $X=(x_1,\cdots, x_{k-1}), W=(z_1,\cdots, z_{k-1})\in \mathbb{H}^{k-1}$ and $z_k\in\mathbb{H}$,
$$\langle X, \varepsilon\ast_p W\rangle=\mathrm{Re}(\alpha)|\langle X, W\rangle_p|^2,$$
where $\alpha=(1+\overline{z_k})^{-1}$ and $\varepsilon=\langle X, W\rangle_p\alpha$.
\end{itemize}
\end{lem}

\begin{proof}
(1) is clear. For (2),
\begin{eqnarray*}
\langle X, \varepsilon \ast_p W\rangle&=&\mathrm{Re}\langle (x_1,\cdots, x_{k-1}), (\varepsilon z_1,\cdots, \varepsilon z_p, \overline{\varepsilon} z_{p+1},\cdots, \overline{\varepsilon} z_{k-1})\rangle_{\mathbb{H}}\nonumber\\
&=&\mathrm{Re}(x_1\overline{\varepsilon}\overline{z_1}+\cdots+x_{p}\overline{\varepsilon}\overline{z_{p}}+\overline{\varepsilon}z_{p+1}\overline{x}_{p+1}+\cdots+
\overline{\varepsilon}z_{k-1}\overline{x}_{k-1})\nonumber\\
&=&\mathrm{Re}(\overline{\varepsilon}\langle X, W\rangle_p)\nonumber\\
&=&\mathrm{Re}(\overline{\alpha}\overline{\langle X, W\rangle_p}\langle X, W\rangle_p)\nonumber\\
&=&\mathrm{Re}(\alpha)|\langle X, W\rangle_p|^2.
\end{eqnarray*}
During the proof, the following identity is used
$$\mathrm{Re}(xy)=\mathrm{Re}(yx), \;\textrm{for} ~ x, y \in \mathbb{H}.$$
Now the proof of Lemma \ref{2.1} is complete.
\end{proof}

It follows from (1) of Lemma \ref{2.1} that $\langle Y, z\rangle_p=0$, and from (2) of Lemma \ref{2.1}
that $|Y|^2=1$. Hence, $\psi_1$ is well-defined.

Next, construct another embedding
$$\psi_2: (S^{4k-1}\setminus \{N\})\times S^{4k-5}\rightarrow N_+, \quad \psi_2(z, X)=(z, Y),$$
where $$X\in \mathbb{H}^{k-1}\cong\mathbb{H}^{k-1}\oplus ~0\subset \mathbb{H}^{k}, ~~\,\,W=(z_1,\cdots, z_{k-1}),$$
$$b'=(1-z_k)(1-\overline{z_k})^{-1}, ~~\,\,\alpha'=(1-\overline{z_k})^{-1},$$
and
$$Y=\Big(X-(\langle X, W\rangle_p\alpha')\ast_p W, ~~\langle X, W\rangle_p b'\Big).$$
Similarly, $\psi_2$ is well-defined.

Based on $\psi_1, \psi_2$, according to the definition in \cite{St51}, we get the characteristic map
\begin{eqnarray*}
\chi:~~ S^{4k-2}&\rightarrow& \mathrm{SO}(4k-4)\\
 z&\mapsto& \chi(z),
\end{eqnarray*}
where $\chi(z)$ is defined by
$$\chi(z)(X):=X-2\Big(\langle X, W\rangle_p\frac{1}{(1+z_k)^2}\Big)\ast_p W,~~~ \text{for}~X\in \mathbb{H}^{k-1}.$$
In particular, for the definite case $p=k-1$, regarding $\mathrm{Sp}(k-1)$ as a subgroup of $\mathrm{U}(2k-2)$ and in turn a subgroup of $\mathrm{SO}(4k-4)$, this is just the characteristic map (\ref{10}) by Steenrod.

Under the standard identification $\mathbb{H}^{k-1}\cong \mathbb{R}^{4k-4}$, for each $z\in S^{4k-2}$, it follows that $\chi(z)$ is
a real linear transformation and preserves the real inner product. In particular, $\chi(0, z_k)=I_{4k-4}$, $\forall ~(0, z_k)\in S^{4k-2}$. Hence, for each $z\in S^{4k-2}$, $\chi(z)\in \mathrm{SO}(4k-4)$.
Therefore, $\chi$ is well-defined. More precisely, for $X=(X_1, X_2)\in \mathbb{H}^{p}\oplus \mathbb{H}^{k-1-p}$,
$$\chi(z)(X)=\left(X_1-(X_1\overline{W_1}^t+W_2\overline{X_2}^t)\beta W_1, ~~~X_2-\overline{\beta}(W_1\overline{X_1}^t+X_2\overline{W_2}^t)W_2\right),$$
where $z=(W, z_k)\in S^{4k-2}$, $W=(W_1, W_2)\in \mathbb{H}^{p}\oplus \mathbb{H}^{k-1-p}$ and $\beta=2(1+z_k)^{-2}$.
\end{proof}

\begin{rem}\label{1.1}\rm{
When $m=4$, $k=2$ and $p=0$, it holds that $\mathrm{Tr}(P_0P_1P_2P_3P_4)=0$ by Proposition \ref{prop4}, which corresponds to the inhomogeneous case with $(m_1,m_2)=(4,3)$. In this case, the characteristic map is given by
\begin{eqnarray*}
\chi: ~~S^6&\longrightarrow& \mathrm{Sp}(1)\\
z=(z_1, z_2) &\mapsto& \chi(z_1, z_2),
\end{eqnarray*}
where $\mathrm{Re}(z_2)=0$ and $\chi(z_1, z_2)$ is defined by
$$\chi(z_1, z_2)(y)=-\frac{(1+z_2)^2}{(1-z_2)^2}y, ~~\forall~ y\in \mathbb{H}.$$
Define
$$F: S^6\times [0, 1]\rightarrow \mathrm{Sp}(1), ~~(z_1, z_2, t)\mapsto-\frac{(1+tz_2)^2}{(1-tz_2)^2}.$$
Since $|F(z_1, z_2, t)|=1$ for any $(z_1, z_2)\in \mathbb{H}\oplus \mathrm{Im}\mathbb{H}$ and $t\in [0, 1]$, it is clear
that $F$ is well-defined.
Consequently, $\chi$ is homotopic to a constant map, and thus the sphere bundle
$S^{3}\hookrightarrow M_+^{10}\rightarrow S^{7}$ is trivial, $M_+\cong S^7\times S^3$. This is consistent with the result in Theorem 2.3 of \cite{QTY21}.}
\end{rem}

\begin{rem}\label{indef 4 rem}\rm{
In the case with $(m_1,m_2)=(4,4k-5)$ and $0\leq p\leq k-1$, the characteristic map is
$\chi: S^{4k-2} \rightarrow \mathrm{SO}(4k-4),
 z \mapsto \chi(z),$
where $\mathrm{Re}(z_k)=0$ and $\chi(z)$ is defined by
$$\chi(z)(X):=X-2\Big(\langle X, W\rangle_p\frac{1}{(1+z_k)^2}\Big)\ast_p W,~~~ \text{for}~X\in \mathbb{H}^{k-1}.$$
Letting $X=(1,0,\cdots,0)\in S^{4k-5}$, for $W=(z_1,\cdots,z_{k-1})$, we have
\begin{equation*}
\langle X, W\rangle_p=\left\{\begin{array}{ll}
\overline{z_1}, ~~p\geq 1\\
z_1, ~~p=0.
\end{array}\right.
\end{equation*}
Combining with the natural projection $\pi: SO(4k-4)\rightarrow S^{4k-5}$ induced by $X=(1,0,\cdots,0)\in S^{4k-5}$,
we obtain
$$\sigma_{k, p}:=\pi\circ\chi:~~S^{4k-2}\longrightarrow S^{4k-5},$$
which is defined as
\begin{eqnarray}\label{sigma p}
&&\\
&&\sigma_{k, p}(z)=\small{\left\{\begin{array}{ll}
\Big(1-2\overline{z_1}\frac{1}{(1+z_k)^2}z_1, -2\overline{z_1}\frac{1}{(1+z_k)^2}z_ 2,\cdots,-2\overline{z_1}\frac{1}{(1+z_k)^2}z_ p,\\
\hspace{2.8cm} -2\frac{1}{(1-z_k)^2}z_1z_{p+1},
\cdots, -2\frac{1}{(1-z_k)^2}z_1z_{k-1}\Big),~~p\geq 1\vspace{3mm}\\
\Big(1-2\frac{1}{(1-z_k)^2}|z_1|^2, -2\frac{1}{(1-z_k)^2}\overline{z_1}z_2,\cdots,-2\frac{1}{(1-z_k)^2}\overline{z_1}z_{k-1}\Big),~~p=0.
\end{array}\right.}\nonumber
\end{eqnarray}
where $z=(W, z_k)$ and $\mathrm{Re}(z_k)=0$.}
\end{rem}

For the definite case with $p=k-1$, we have dealt with
the homotopy class $[\sigma_{k, p}]=[\sigma_k]$ in Remark \ref{1.3}. From now on, we will focus on the indefinite case, i.e., the case with $p\neq k-1$.
We start with an algebraic fact.
\begin{lem}\label{extend}
When $m=4$, the symmetric Clifford system $\{P_0,\cdots,P_4\}$ of indefinite type on $\mathbb{R}^{8k}$ defined in (\ref{P}) can be extended to a symmetric Clifford system $\{P_0,\cdots, P_5\}$ if and only if $k=2p+2$.
\end{lem}
\begin{proof}
For the sufficient part, it is mentioned in Remark 3.2 of \cite{QT16} that when $p=0$ and $k=2$, the indefinite Clifford system $\{P_0,\cdots,P_4\}$ on $\mathbb{R}^{16}$ can be extended to a symmetric Clifford system $\{P_0,\cdots, P_5\}$. To be precise,
for $0\leq p< k-1$ and $k=2p+2$, we construct $E_4$ explicitly as follows
$$E_4(z)=(z_{p+1},\cdots,z_{2p},-z_1,\cdots,-z_p, z_{2p+2}, -z_{2p+1}).$$
Recalling the linear transformations $E_i$ $( i=1,2,3)$ defined in (\ref{E}), 
it is easy to check that $E_4E_i+E_iE_4=0$. Thus $P_5(z,w)=(E_4w,-E_4z)$, combining with $P_0,\cdots,P_4$ in (\ref{P}), constitutes a symmetric Clifford system on $\mathbb{R}^{8k}$.

Now we need only to consider the necessary part.
Denote the skew-symmetric matrices corresponding to $E_1, E_2, E_3$ by $\overline{E}_1, \overline{E}_2, \overline{E}_3$, i.e., $E_i(z)=z\overline{E}_i$ $(i=1,2,3)$. Thus $\overline{E}_1, \overline{E}_2, \overline{E}_3$ are expressed as
\begin{equation*}
\tiny{
\overline{E}_1=\left(\begin{smallmatrix}
D_1&&&&&&\\
&\ddots&&&&&\\
&&D_1&&&&\\
&&&&&&\\
&&&-D_1&&&\\
&&&&\ddots&&\\
&&&&&-D_1&\\
&&&&&&\\
&&&&&&D_1
\end{smallmatrix}\right),
~~\overline{E}_2=\left(\begin{smallmatrix}
D_2&&&&&&\\
&\ddots&&&&&\\
&&D_2&&&&\\
&&&&&&\\
&&&-D_2&&&\\
&&&&\ddots&&\\
&&&&&-D_2&\\
&&&&&&\\
&&&&&&D_2
\end{smallmatrix}\right),}
\end{equation*}

\begin{equation*}
\tiny{
\overline{E}_3=\left(\begin{smallmatrix}
D_3&&&&&&\\
&\ddots&&&&&\\
&&D_3&&&&\\
&&&&&&\\
&&&-D_3&&&\\
&&&&\ddots&&\\
&&&&&-D_3&\\
&&&&&&\\
&&&&&&D_3
\end{smallmatrix}\right),}
\end{equation*}
where
$D_1=  \left(\begin{smallmatrix}
0& 1 & & \\
-1& 0 & &\\
 & & 0 & 1 \\
 & & -1& 0
\end{smallmatrix}\right)
$,~ $D_2=  \left(\begin{smallmatrix}
 & & 1& 0  \\
 & & 0 & -1  \\
 -1& 0&  &  \\
 0&1 & &
\end{smallmatrix}\right)
$~and~~ $D_3=  \left(\begin{smallmatrix}
 & & 0& 1  \\
 & & 1 & 0  \\
 0& -1 &  &  \\
 -1& 0 & &
\end{smallmatrix}\right)$
denote the matrices corresponding to the left multiplication by $\mathrm{i}, \mathrm{j}$ and $\mathrm{k}$, respectively.

Suppose $\{P_0,\cdots,P_4\}$ on $\mathbb{R}^{8k}$ can be extended to $\{P_0,\cdots, P_4, P_5\}$. Then there exists an skew-symmetric matrix $\overline{E}_4$ such that $\{P_0,\cdots, P_4, P_5\}$ can be defined in the same way as in (\ref{P}). Denote the skew-symmetric matrix $\overline{E}_4$ by
$\overline{E}_4=\left(\begin{array}{cc|c}
A& *&  b_1 \\
* & B & b_2\\
\hline
-b_1^t &-b_2^t & a
\end{array}\right),
$
where $A$ is a skew-symmetric $4p\times 4p$ matrix, $B$ is a skew-symmetric $4(k-1-p)\times 4(k-1-p)$ matrix, $a$ is a skew-symmetric $4\times 4$ matrix, 
$b_1$ is a $4p\times 4$ matrix, and $b_2$ is a $4(k-1-p)\times 4$ matrix. Then from $(\overline{E}_4\overline{E}_3\overline{E}_2\overline{E}_1)^t=\overline{E}_1^t\overline{E}_2^t\overline{E}_3^t\overline{E}_4^t=\overline{E}_1\overline{E}_2\overline{E}_3\overline{E}_4=\overline{E}_4\overline{E}_3\overline{E}_2\overline{E}_1$, it follows that
$$\overline{E}_4\overline{E}_3\overline{E}_2\overline{E}_1=\left(\begin{array}{cc|c}
A& *& b_1 \\
*& -B& b_2\\
\hline -b_1^t&b_2^t & a
\end{array}\right)$$
is symmetric, which implies that $b_1=0$ and $A^t=A$, $B^t=B$, $a^t=a$. Thus $A=0$, $B=0$, $a=0$, since $A$, $B$, $a$ are simultaneously
skew-symmetric.
Then combining with the property that $\overline{E}_4$ is non-degenerate,
we conclude $k=2p+2$.
\end{proof}

Now we are in a position to prove Proposition \ref{quaternion}. During the proof, the homotopy class $[\sigma_{k, p}]$ will be also determined.
\vspace{2mm}

\noindent \textbf{Proposition \ref{quaternion}.} \emph{Let $m=4$ and $\mathrm{Tr}(P_0P_1P_2P_3P_4)=-8(2p-k+2)$ with $0\leq p\leq k-1$.
The sphere bundle $S^{4k-5}\hookrightarrow M_+^{8k-6}\rightarrow S^{4k-1}$ associated with the vector bundle $\eta$ admits a cross-section if and only if $k-2-2p$ can be divided by $24$.}
\vspace{2mm}

\begin{proof}
As an immediate consequence of Lemma \ref{extend}, in the case with $k=2p+2$, the fact that the indefinite $\{P_0,\cdots,P_4\}$ can be extended to $\{P_0,\cdots, P_5\}$ implies that the sphere bundle $S^{4k-5}\hookrightarrow M_+^{8k-6}\rightarrow S^{4k-1}$ admits a cross-section, thus $\sigma_{k,p}$ is homotopic to a constant map, and $[\sigma_{k, p}]=0$ in $\pi_{4k-2}S^{4k-5}$.  We should also mention that a definite Clifford system $\{P_0,\cdots,P_4\}$ on $\mathbb{R}^{8k}$ could not extend to $\{P_0,\cdots, P_5\}$. However, even if a Clifford system $\{P_0,\cdots,P_4\}$ is not extendable, a cross-section might still exist.

For $p=0$ and general $k$, recall $\sigma_{k, 0}: S^{4k-2}\rightarrow S^{4k-5}$ in (\ref{sigma p}) defined by
$$\sigma_{k, 0}(z_1, z_2,\cdots, z_k)= \Big(1-2(1-z_k)^{-2}|z_1|^2,~~-2(1-z_k)^{-2}\overline{z_1}z_2,\cdots, -2(1-z_k)^{-2}\overline{z_1}z_{k-1}\Big).
$$
Based on a similar consideration in Remark \ref{1.3}, we define a continuous function
\begin{eqnarray*}
f:~S^3\times S^{4k-6}&\longrightarrow& S^{4k-6}\\
z_1, (z_2,\cdots, z_k)&\mapsto& (z_k, -\overline{z_1}z_2,\cdots,-\overline{z_1}z_{k-1}),
\end{eqnarray*}
where $\mathrm{Re}(z_k)=0$. The corresponding Hopf construction is
\begin{eqnarray*}
\qquad H(f):~~ S^{4k-2}&\longrightarrow& S^{4k-5}\\
\qquad (z_1, z_2,\cdots, z_k)&\mapsto& (1-2|z_1|^2+2|z_1|z_k, -2\overline{z_1}z_2,\cdots,-2\overline{z_1}z_{k-1}).
\end{eqnarray*}
As is well known, to obtain a homotopy $F(t)$ between two maps $f, g$ into a sphere, one needs to take the point dividing the shortest geodesic (great circle arc) joining $f(z)$ and $g(z)$ into a ratio $t:1-t$. To do this, we need $f(z)\neq -g(z)$ for every $z$.  We can see that $(\sigma_{k, 0}+H(f))(z)\neq 0$ for any $z\in S^{4k-1}$ with $\mathrm{Re}(z_k)=0$. For clarity, we make a short explanation as follows and use proof by contradiction. Suppose $(\sigma_{k, 0}+H(f))(z)=0$ for certain $z\in S^{4k-1}$ with $\mathrm{Re}(z_k)=0$. Then we have
\begin{eqnarray}
&&2-2|z_1|^2(\frac{1-|z_k|^2}{(1+|z_k|^2)^2}+1)-2|z_1|(\frac{2|z_1|}{(1+|z_k|^2)^2}+1)z_k=0, \label{1}\\
&&(\frac{1}{(1-z_k)^2}+1)\overline{z_1}z_2=\cdots=(\frac{1}{(1-z_k)^2}+1)\overline{z_1}z_{k-1}=0.\label{2}
\end{eqnarray}
Since $\frac{1}{(1-z_k)^2}+1\neq 0$,
(\ref{2}) implies that $z_1=0$ or $z_2=\cdots=z_{k-1}=0$. However, it follows from (\ref{1}) that $z_1\neq 0$.
Thus $z_2=\cdots=z_{k-1}=0$. Furthermore, from $z_1\neq 0$, $\mathrm{Re}(z_k)=0$ and (\ref{1}), we obtain that $z_k=0$.
From $z\in S^{4k-1}$ and $z_2=...=z_k=0$, it follows that $|z_1|=1$. On the other hand, (\ref{1}) and $z_k=0$
imply that $2|z_1|^2=1$. Now, we get a contradiction.
Therefore, $\sigma_{k, 0}\simeq H(f)$, and thus $$[\sigma_{k, 0}]=[H(f)]\equiv -(k-2) ~(\mathrm{mod}~ 24)$$
in $\pi_{4k-2}S^{4k-5}\cong Z_{24}.$

For $1\leq p< k-1$, recall $\sigma_{k, p}: S^{4k-2}\rightarrow S^{4k-5}$ in (\ref{sigma p}) defined by
\begin{eqnarray*}
\sigma_{k, p}(z)&=&\Big(1-2\overline{z_1}(1+z_k)^{-2}z_1, -2\overline{z_1}(1+z_k)^{-2}z_ 2,\cdots,-2\overline{z_1}(1+z_k)^{-2}z_ p,\\
&&\hspace{3.5cm}-2(1-z_k)^{-2}z_1z_{p+1}, \cdots, -2(1-z_k)^{-2}z_1z_{k-1}\Big).
\end{eqnarray*}
Define $\theta:~S^{4k-2}\longrightarrow S^{4k-5}$ by
$$\theta(z)=\Big(1-2\overline{z_1}(1+z_k)^{-2}z_1, -\frac{2\overline{z_1}z_2}{1+|z_k|^2}, \cdots, -\frac{2\overline{z_1}z_p}{1+|z_k|^2}, -\frac{2z_1z_{p+1}}{1+|z_k|^2}, \cdots,-\frac{2z_1z_{k-1}}{1+|z_k|^2}\Big).
$$
Similar to the case $p=0$, using proof by contradiction, one can easily see that
$(\sigma_{k, p}+\theta)(z)\neq 0$ for any $z\in S^{4k-1}$ with $\mathrm{Re}(z_k)=0$.  Thus
$\sigma_{k, p}\simeq \theta$.

On the other hand, define  $\phi:~S^{4k-2}\longrightarrow S^{4k-5}$ by
$$\phi(z)=\Big(1-2|z_1|^2+2\frac{\overline{z_1}z_kz_1}{|z_1|}, -2\overline{z_1}z_2,\cdots, -2\overline{z_1}z_p, -2z_1z_{p+1}, \cdots, -2z_1z_{k-1}\Big).$$
Again, using proof by contradiction, one gets
$(\phi+\theta)(z)\neq 0$ for any $z\in S^{4k-1}$ with $\mathrm{Re}(z_k)=0$. Thus
$\phi\simeq\theta\simeq \sigma_{k, p}$.

Consequently, by virtue of the discussion in Remark \ref{1.3}, we obtain
\begin{eqnarray*}
[\sigma_{k, p}]&=&[\theta]=[\phi]\\
&=&[H(\omega_0)]+[H(\omega_1)]+\cdots+[H(\omega_p)]+(k-1-p)[\Sigma^{4k-9}H(\zeta)]\\
&=&-[\Sigma^{4k-8}H(\varrho)]-(p-1)[\Sigma^{4k-9}H(\zeta)]+(k-1-p)[\Sigma^{4k-9}H(\zeta)]\\
&\equiv& k-2-2p~(\mathrm{mod}~24).
\end{eqnarray*}

Therefore, the sphere bundle $S^{4k-5}\hookrightarrow M_+^{8k-6}\rightarrow S^{4k-1}$ admits a cross-section if and only if $k-2-2p$ can be divided by $24$.
\end{proof}

\subsubsection{}\hspace{-2.5mm}\textbf{Natural questions.}\label{Natural questions}
For $0\leq p\leq k-1$, define
$$N_{p, k, q}(\mathbb{H}):=\{ A=(a_1^t,\cdots, a_q^t)^t\in M_{q\times k}(\mathbb{H})~|~a_i\in \mathbb{H}^k, \langle a_i, a_j\rangle_{p}=\delta_{ij}, i,j=1,\cdots,q\}.$$
Clearly, $N_{k-1, k, k}(\mathbb{H})$=$\mathrm{Sp}(k)$. The natural questions are as follows:

(1). For $0\leq p\leq k-2$, does $N_{p, k, q}(\mathbb{H})$ admit a smooth manifold structure?

(2). For $0\leq p\leq k-2$, if (1) is true, does $N_{p, k, q}(\mathbb{H})$ admit a smooth metric with nice curvature property?
\begin{rem}\rm{
1). If $k=2$ and $p=0$, then $N_{0, 2, 2}(\mathbb{H})\cong S^3\times S^7$.

2). Very recently, the authors proved that $N_{p, k, 3}(\mathbb{H})$ admits a smooth manifold structure in \cite{QTY22}. However, the topological and geometric properties of such manifolds still need to be deeply explored.}
\end{rem}


\subsection{The $(m_1, m_2)=(8, 8k-9)$ case}
For $m=8$ and $l=8k$, let $\{P_0, P_1,\cdots, P_8\}$ be a symmetric Clifford system on $\mathbb{R}^{16k}$. The definite and indefinite cases will be dealt with in subsections \ref{def 8} and \ref{indef 8}, respectively.
\subsubsection{}\hspace{-2.5mm}\textbf{The definite case.}\label{def 8}
In this subsection, we consider the definite case with
$$P_0P_1\cdot\cdot\cdot P_8=\pm I_{16k}.$$

For $(u,v)\in \mathbb{O}^k\oplus\mathbb{O}^k=\mathbb{R}^{16k}$, write $u=(u_1, \cdots, u_k)$ and $v=(v_1, \cdots, v_k)$. Construct a symmetric Clifford system $\{P_0,\ldots, P_8\}$ on $\R^{16k}$ as follows:
\begin{equation}\label{clifford system}
P_0(u ,v)=(u, -v),\quad P_1(u, v)=(v, u),\quad P_{1+\alpha}(u, v)=(E_{\alpha}v, -E_{\alpha}u),
\end{equation}
where for the standard orthonormal basis $\{1, e_1, \cdots, e_7\}$ of the octonions $\mathbb{O}$ (Cayley numbers), $E_{\alpha}$ acts on $u$ or $v$ by
\begin{equation*}\label{E action}
E_{\alpha}u=(e_{\alpha}u_1, e_{\alpha}u_{2},\cdots, e_{\alpha}u_{k}),~\alpha=1,\cdots,7.
\end{equation*}

In fact, $1=(1, 0)$, $e_1=(\mathrm{i}, 0)$, $e_2=(\mathrm{j}, 0)$, $e_3=(\mathrm{k}, 0)$, $e_4=(0, 1)$,
$e_5=(0, \mathrm{i})$, $e_6=(0, \mathrm{j})$, $e_7=(0, \mathrm{k})\in\mathbb{H}\oplus\mathbb{H}$.
Recalling the Cayley-Dickson construction of the product in octonions $\mathbb{O}\cong \mathbb{H}\oplus\mathbb{H}$:
\begin{eqnarray*}
\mathbb{O}\times \mathbb{O}&\longrightarrow&\mathbb{O}\\
(a,b),~(c,d)&\mapsto&(a,b)\cdot (c,d) := (ac-\bar{d}b, ~ da+b\bar{c}),
\end{eqnarray*}
one can see easily that (cf. \cite{GTY20})
$$ e_1(e_2(\cdots(e_7z)\cdots))=-z, \;\textrm{for any} \;z\in\mathbb{O}.$$
Moreover, $(a,0)(c,d)=(ac,da)$ and $(0,b)(c,d)=(-\bar{d}b, b\bar{c})$ for any $a,b,c,d\in\mathbb{H}$.
Taking $a=i, j$ or $k$, $b=1, i, j,$ or $k$, for the standard orthonormal basis 
above, one can see easily that
$$ e_{\alpha}(e_{\beta}z)=-e_{\beta}(e_{\alpha}z) ~~(\alpha\neq \beta), ~~\text{for any}~ z\in \mathbb{O}.$$
Thus $\{P_0,\ldots, P_8\}$ forms a symmetric Clifford system on $\R^{16k}$. Furthermore, a direct computation
leads to $P_0\cdots P_8=-I_{16k}$, which indicates that the system is definite.

For $z=(z_1,\cdots, z_k), w=(w_1,\cdots, w_k)\in \mathbb{O}^{k}$, define
\begin{equation}\label{H}
\langle z, w\rangle_{\mathbb{O}}:=\sum_{i=1}^{k}z_i\overline{w}_i.
\end{equation}
Clearly, the Euclidean inner product $ \langle z, w\rangle=\mathrm{Re}    \langle z, w\rangle_{\mathbb{O}}.$

Now, we are ready to give a proof to the following
\begin{prop}\label{1.2}
Let $M_+\subset \mathbb{O}^k\oplus\mathbb{O}^k=\mathbb{R}^{16k}$ be the focal submanifold of OT-FKM type with $(g, m_1, m_2)=(4, 8, 8k-9)$ in the definite case. Then $(z, w)\in M_+$ if and only if
$$|z|=|w|=\frac{1}{\sqrt{2}}, \langle z, w\rangle_{\mathbb{O}}=0.$$
\end{prop}

\begin{proof}
Let $M_+$ be the focal submanifold of OT-FKM type with respect to the symmetric Clifford system $\{P_0,\ldots, P_8\}$ given in (\ref{clifford system}).
Using the well-known identity $\mathrm{Re}(\lambda(\mu \nu))=\mathrm{Re}((\lambda \mu)\nu)$ for any $\lambda,\mu,\nu\in \mathbb{O}$, we obtain that
\begin{eqnarray*}
\langle z, E_{\alpha}w\rangle&=&\mathrm{Re}\langle z, E_{\alpha}w\rangle_{\mathbb{O}}=\mathrm{Re}(\sum_{i=1}^{k}z_i(\overline{e_{\alpha}w_i}))=\mathrm{Re}(\sum_{i=1}^{k}(z_i\overline{w}_i)\overline{e}_{\alpha})\\
&=&\mathrm{Re}\langle \langle z, w\rangle_{\mathbb{O}}, e_{\alpha}\rangle_{\mathbb{O}}\\
&=&\langle \langle z, w\rangle_{\mathbb{O}}, e_{\alpha}\rangle\quad \text{for} ~~\alpha=1,\cdots,7,
\end{eqnarray*}
and further
$$\langle P_{1+\alpha}(z, w), (z, w)\rangle=\langle (E_{\alpha}w, -E_{\alpha}z), (z,w)\rangle=2\langle \langle z, w\rangle_{\mathbb{O}}, e_{\alpha}\rangle, \quad \text{for} ~~\alpha=1,\cdots,7.$$
Therefore, by definition,  the focal submanifold $M_+$ associated
with the given Clifford system is determined to be
$$M_+=\{(z, w)\in \mathbb{O}^{k}\oplus\mathbb{O}^{k}~|~|z|=|w|=\frac{1}{\sqrt{2}}, \langle z, w\rangle_{\mathbb{O}}=0\} $$
\end{proof}

\begin{rem}\label{1.4}\rm{
By Proposition \ref{1.2}, the focal submanifold $M_+$ of OT-FKM type with respect to the definite Clifford system $\{P_0,\cdots, P_8\}$
on $\mathbb{R}^{16k}$ is exactly diffeomorphic to the manifold $V_{2}(\mathbb{O}^k)$ defined by \cite{Ja58}.}
\end{rem}
In \cite{Ja58}, James defined $V_l(\mathbb{O}^k)$, the space of orthonormal $l$-frames in $\mathbb{O}^{k}$. More precisely,
$$V_l(\mathbb{O}^k)=\{(a_1,\cdots ,a_l)\in \mathbb{O}^k\oplus \cdots\oplus
\mathbb{O}^k~|~\langle a_i, a_j\rangle_{\mathbb{O}}=\delta_{ij}, for \;1\leq i,j\leq l    \}.$$
For example, $V_1(\mathbb{O}^k)$ is an $(8k-1)$-sphere. He defined a projection of $V_{l_1}(\mathbb{O}^k)$ into $V_{l_2}(\mathbb{O}^k)$, where $1\leq l_2\leq l_1$ by suppressing the first $l_1-l_2$ vectors of each $k$-frame, and asked if this is a fiber map, and if $V_l(\mathbb{O}^k)$ is a manifold. He showed that it is true for the projection $p: V_{2}(\mathbb{O}^k)\rightarrow V_1(\mathbb{O}^k)$, and thus $V_{2}(\mathbb{O}^k)$ admits a manifold structure. As explained in Remark \ref{1.4}, $V_{2}(\mathbb{O}^k)$ is indeed a manifold diffeomorphic to $M_+$ of OT-FKM type with respect to the definite Clifford system $\{P_0,\cdots, P_8\}$. For recent progress on the two questions of James, we refer to \cite{QTY22}.

In the following, we will construct the characteristic map of $V_{2}(\mathbb{O}^k)$ as an $S^{8k-9}$-bundle over $S^{8k-1}$.
As before, $S^{8k-1}$ is the unit sphere in $\mathbb{O}^{k}$ and $S^{8k-2}$ is its equator defined by $\mathrm{Re}(z_k)=0$. Let $N=(0,\cdots,0, 1)\in \mathbb{O}^{k}$. For
$$V_{2}(\mathbb{O}^k)=\{(z, w)\in \mathbb{O}^k\oplus\mathbb{O}^k~|~\langle z, z\rangle_{\mathbb{O}}=\langle w, w\rangle_{\mathbb{O}}=1, \langle z, w\rangle_{\mathbb{O}}=0\},$$
we first construct an embedding
$$\psi_1: (S^{8k-1}\setminus \{-N\})\times S^{8k-9}\rightarrow V_{2}(\mathbb{O}^k), \quad \psi_1(z, X)=(z, Y),$$
where $$X\in \mathbb{O}^{k-1}\cong\mathbb{O}^{k-1}\oplus 0\subset \mathbb{O}^{k}, W=(z_1,\cdots, z_{k-1}),$$
$$b=(1+z_k)(1+\overline{z_k})^{-1}, \alpha=(1+\overline{z_k})^{-1},$$
and
$$Y=\Big(X-(\langle X, W\rangle_{\mathbb{O}}\alpha)W, ~~-\langle X, W\rangle_{\mathbb{O}} b\Big).$$
To verify the map $\psi_1$ is well-defined, we need to show that $\langle Y, z\rangle_{\mathbb{O}}=0$ and $|Y|^2=1$. Firstly, a direct calculation leads to
\begin{eqnarray*}
\langle Y, z\rangle_{\mathbb{O}}&=&\langle X, W\rangle_{\mathbb{O}}-\langle \left(\langle X, W\rangle_{\mathbb{O}}\alpha\right)W, W\rangle_{\mathbb{O}}-\left(\langle X, W\rangle_{\mathbb{O}}b\right)\bar{z}_k\\
&=& \langle X, W\rangle_{\mathbb{O}}-\sum_{i=1}^{k-1}\Big(\left(\langle X, W\rangle_{\mathbb{O}}\alpha\right)z_i\Big)\bar{z_i}-\left(\langle X, W\rangle_{\mathbb{O}}(\frac{1+z_k}{1+\bar{z}_k})\right)\bar{z}_k\\
&=& \langle X, W\rangle_{\mathbb{O}}-\left(\langle X, W\rangle_{\mathbb{O}}\alpha\right)|W|^2-\langle X, W\rangle_{\mathbb{O}}\Big((\frac{1+z_k}{1+\bar{z}_k})\bar{z}_k\Big).
\end{eqnarray*}
As for the last equality, we have used the well-known formula that $(ab)\bar{b}=a|b|^2$ for any $a,b\in \mathbb{O}$, as well as  the Artin theorem which states that any subalgebra of $\mathbb{O}$ generated by two linearly independent elements is associative.
It is clear that $1-\alpha|W|^2-(\frac{1+z_k}{1+\bar{z}_k})\bar{z}_k=0$ , thus $\langle Y, z\rangle_{\mathbb{O}}=0$.
Next,
\begin{eqnarray*}
|Y|^2&=&|X-(\langle X, W\rangle_{\mathbb{O}}\alpha)W|^2+|\langle X, W\rangle_{\mathbb{O}} b|^2\\
&=&|X|^2+(|\alpha|^2|W|^2+|b|^2)\cdot|\langle X, W\rangle_{\mathbb{O}}|^2-2\langle X, (\langle X, W\rangle_{\mathbb{O}}\alpha)W\rangle\\
&=& 1+(|\alpha|^2|W|^2+|b|^2)\cdot|\langle X, W\rangle_{\mathbb{O}}|^2-2\mathrm{Re}\langle X, (\langle X, W\rangle_{\mathbb{O}}\alpha)W\rangle_{\mathbb{O}}\\
&=& 1+(|\alpha|^2|W|^2+|b|^2)\cdot|\langle X, W\rangle_{\mathbb{O}}|^2-2\mathrm{Re}\sum_{i=1}^{k-1}x_i(\bar{z_i}\overline{\langle X, W\rangle_{\mathbb{O}}\alpha})\\
&=& 1+(|\alpha|^2|W|^2+|b|^2)\cdot|\langle X, W\rangle_{\mathbb{O}}|^2-2\mathrm{Re}\sum_{i=1}^{k-1}(x_i\bar{z_i})\overline{\langle X, W\rangle_{\mathbb{O}}\alpha}\\
&=& 1+(|\alpha|^2|W|^2+|b|^2)\cdot|\langle X, W\rangle_{\mathbb{O}}|^2-2\mathrm{Re}\langle X, W\rangle_{\mathbb{O}}(\bar{\alpha}\overline{\langle X, W\rangle_{\mathbb{O}}})\\
&=&1+(|\alpha|^2|W|^2+|b|^2-2\mathrm{Re}\alpha)\cdot|\langle X, W\rangle_{\mathbb{O}}|^2\\
&=&1,
\end{eqnarray*}
where we have used a well-known formula that $\langle ab, c \rangle=\langle a, c\bar{b} \rangle$ for any $a, b, c\in \mathbb{O}$ and a trivial equality  $|\alpha|^2|W|^2+|b|^2-2\mathrm{Re}\alpha=0$.
Consequently, $\psi_1$ is well-defined.

Moreover, we can construct another well-defined embedding
$$\psi_2: (S^{8k-1}\setminus \{N\})\times S^{8k-9}\rightarrow N_+, \quad \psi_2(z, X)=(z, Y),$$
where $$X\in \mathbb{O}^{k-1}\cong\mathbb{O}^{k-1}\oplus 0\subset \mathbb{O}^{k}, ~~W=(z_1,\cdots, z_{k-1}),$$
$$b'=(1-z_k)(1-\overline{z_k})^{-1}, ~~\alpha'=(1-\overline{z_k})^{-1},$$
and
$$Y=\Big(X-(\langle X, W\rangle_{\mathbb{O}}\alpha')W, ~~\langle X, W\rangle_{\mathbb{O}} b'\Big).$$

Furthermore, we get the characteristic map
$$\chi: S^{8k-2}\rightarrow \mathrm{SO}(8k-8),\quad  z\mapsto \chi(z),$$
where $\chi(z)$ is defined by
$$\chi(z)(X):=X-2(\langle X, W\rangle_{\mathbb{O}}\cdot(1+z_k)^{-2})W.$$
Under the standard identification $\mathbb{O}^{k-1}\cong \mathbb{R}^{8k-8}$, for each $z\in S^{8k-2}$, it follows that $\chi(z)$ is
a real linear transformation and preserves the real inner product. In particular, $\chi(0, z_k)=I_{8k-8}$, $\forall z=(0, z_k)\in S^{8k-2}$.
Hence, for each $z\in S^{8k-2}$, $\chi(z)\in \mathrm{SO}(8k-8)$.
Therefore, $\chi$ is well-defined.

\begin{rem}\label{1.8}\rm{
When $k=2$, this is just the definite case with $(m_1, m_2)=(8, 7)$. By \cite{FKM81}, the focal submanifold $M_+$ is homogeneously embedded in $S^{31}$. The characteristic map is given by
\begin{eqnarray*}
\chi: S^{14}&\longrightarrow& \mathrm{SO}(8)\\
\chi(z_1, z_2)(X)&=&X-2((X\bar{z}_1)(1+z_2)^{-2})z_1, ~~\forall ~~X\in \mathbb{O},
\end{eqnarray*}
where $\mathrm{Re}(z_2)=0$.  Set $X=(1,0,\cdots,0)\in S^7$. Then
$$\chi(z)X=1-2(\overline{z}_1(1+z_2)^{-2})z_1.$$
Composing with the natural projection $\pi: SO(8)\rightarrow S^7$ induced by $X=(1,0,\cdots,0)\in S^7$,
it gives a map
\begin{eqnarray*}\label{g}
g_2:=\pi\circ\chi: ~~S^{14}&\longrightarrow & S^7\\
z=(z_1, z_2)&\mapsto& 1-2(\overline{z}_1(1+z_2)^{-2})z_1.\nonumber
\end{eqnarray*}
For $z=(z_1,z_2)\in S^{15}\subset \mathbb{O}^2$ with $\mathrm{Re}(z_2)=0$, define
\begin{eqnarray*}
\tau:~S^{14}&\longrightarrow& S^{7}\\
(z_1,z_2)&\mapsto& 2|z_2|^2-1+2\frac{\overline{z_1}z_2z_1}{|z_1|}.
\end{eqnarray*}
Similarly as before, using proof by contradiction, it is easy to see that\\ $(g_2+\tau)(z)\neq 0$ for any $z=(z_1,z_2)\in S^{15}\subset \mathbb{O}^2$ with $\mathrm{Re}(z_2)=0$. Thus $g_2\simeq \tau$.

According to \cite{DMR04}, the following map $b$ generates $\pi_{14}S^7\cong \mathbb{Z}_{120}$,
\begin{eqnarray*}
b:~S^{14}&\longrightarrow& S^{7}\\
(z_1, z_2)&\mapsto& b(z_1,z_2)=\left\{\begin{array}{ll}
\frac{z_1}{|z_1|} e^{\pi z_2} \frac{\overline{z_1}}{|z_1|},~~~ z_1\neq 0\\
~~-1,\qquad\quad  z_1=0,
\end{array}\right.
\end{eqnarray*}
where $\mathrm{Re}(z_2)=0$.
Define $b'(z_1, z_2)=-b(\overline{z_1}, -z_2)$.
If $z_1=0$, it is direct to see that $\tau+b'\neq 0$. We need only to consider the case with $z_1\neq 0$.
Spelling out $e^{\pi z_2}=\cos(\pi|z_2|)+\sin(\pi|z_2|)\frac{z_2}{|z_2|}$, we can express $b'$ as
$$b'(z_1, z_2)=-\cos(\pi|z_2|)+\sin(\pi|z_2|)\frac{\overline{z_1}z_2z_1}{|z_2||z_1|^2}.$$
Thus $(\tau+b')(z)=2|z_2|^2-1-\cos(\pi|z_2|)+(\frac{2}{|z_1|}+\frac{\sin(\pi|z_2|)}{|z_2||z_1|^2})\overline{z_1}z_2z_1$. Suppose $(\tau+b')(z)=0$. Then the imaginary part implies that $z_2=0$, a contradiction. Thus $(\tau+b')(z)\neq 0$ for any $z=(z_1,z_2)\in S^{15}\subset \mathbb{O}^2$ with $\mathrm{Re}(z_2)=0$,
and further, $\tau\simeq b'\simeq b$. Therefore,
$$[g_2]=[\tau]=[b']=[b]\in \pi_{14}S^7$$
is a generator.

When $k\geq 3$, combing with the natural projection
$\pi:~SO(8k-8)\rightarrow S^{8k-9}$ induced by $X=(1, 0,... ,0)\in S^{8k-9}$,
we have a map
$$g_k:=\pi\circ\chi: ~~S^{8k-2}\longrightarrow  S^{8k-9},$$
which is defined by
$$g_k(z_1,\cdots,z_k)=(1-2\overline{z_1}(1+z_k)^{-2}z_1, -2(\overline{z_1}(1+z_k)^{-2})z_2,\cdots,-2(\overline{z_1}(1+z_k)^{-2})z_{k-1}),$$
where $\mathrm{Re}(z_k)=0$.

To investigate the homotopy class $[g_k]\in\pi_{8k-2}S^{8k-9}\cong \pi_7^S\cong \mathbb{Z}_{240}$, we recall some basic knowledge on J-homomorphisms and homotopy groups of spheres.

Define
$$\rho: S^7\rightarrow \mathrm{SO}(7), \quad x\mapsto \rho(x),$$
where $\rho(x)(y)=xy\overline{x}$, $x, y \in S^7\subset \mathbb{O}$ and
$\mathrm{Re}( y)=0$. Moreover, define
$$\sigma: S^7\rightarrow \mathrm{SO}(8), \quad x\mapsto \sigma(x),$$
where $\sigma(x)(y)=xy$, $x \in S^7\subset\mathbb{O}$, and $y\in \mathbb{O}$.

It follows from \cite{TSY57} that $\pi_7(SO(7))\cong \mathbb{Z}$ is generated by $[\rho]$.
By the triviality of the bundle $SO(7)\hookrightarrow SO(8)\rightarrow S^7$, one has a splitting isomorphism
$$\pi_7(SO(8))\cong\pi_7S^7\oplus\pi_7(SO(7))\cong\mathbb{Z}\oplus\mathbb{Z},$$
where the free parts are generated by $[\sigma]$ and $[i_*(\rho)]$, respectively. Here, $i: SO(7)\rightarrow SO(8)$ is the canonical inclusion.
As it is well known, $\pi_7(SO(9))\cong \mathbb{Z}$. Let us consider the canonical inclusion $j: SO(8)\hookrightarrow SO(9)$. It induces a
homomorphism
$$j_*:~\pi_{7}(SO(8))\longrightarrow \pi_{7}(SO(9))\cong\mathbb{Z}.$$
It is shown in \cite{TSY57} that $[j_*\sigma]$ is a generator of $\pi_7(SO(9))$, and $[j_*(i_*\rho)]=2[j_*\sigma]$ in $\pi_7(SO(9))$. Thus,
in $\pi_{16}S^9\cong \mathbb{Z}_{240}$, one has an equality
$$[\Sigma^2H(\rho)]=2[\Sigma H(\sigma)],$$
where $\Sigma$ is the suspension, and $H$ is the Hopf construction, or $J$-homomorphism.
Furthermore, $[\Sigma H(\sigma)]\in \pi_{16}S^9$ is a generator, and thus
$[\Sigma^{p+1}H(\sigma)]\in\pi_{16+p}S^{9+p}\cong\pi_7^S\cong\mathbb{Z}_{240}$ is a generator, $[\Sigma^{p+2}H(\rho)]=2[\Sigma^{p+1}H(\sigma)]$.

Define  $\Phi:~S^{8k-2}\longrightarrow S^{8k-9}$ by
$$\Phi(z)=\Big(1-2|z_1|^2+2\frac{\overline{z_1}z_kz_1}{|z_1|}, -2\overline{z_1}z_2,\cdots, -2\overline{z_1}z_{k-1}\Big),$$
where $\mathrm{Re}(z_k)=0$.
Again, using proof by contradiction, one gets
$(\Phi+g_k)(z)\neq 0$ for any $z\in S^{8k-1}$ with $\mathrm{Re}(z_k)=0$. Thus
$g_k\simeq \Phi$.
In particular, $[H(\rho)]=-[g_2]\in \pi_{14}S^7$.

Choose $\Omega: S^7\rightarrow \mathrm{SO}(8k-9)$ by
$$ \Omega(x_1)(x_2,\cdots, x_k)=(\overline{x_1}x_kx_1, -\overline{x_1}x_2,\cdots, -\overline{x_1}x_{k-1}),$$
where $x_1\in S^7\subset\mathbb{O}$, $(x_2,\cdots,x_{k-1})\in \mathbb{O}^{k-2}$, $x_k\in \mathrm{Im} \mathbb{O}$.
Clearly, $\Phi=H(\Omega)$.

For $i=0$, define $\Omega_0: S^7\rightarrow \mathrm{SO}(8k-9)$ by
$$\Omega_0(x_1)(x_2,\cdots, x_k)=(\overline{x_1}x_kx_1, x_2,\cdots, x_{k-1}),$$
and for $1\leq i\leq k-2$, define $\Omega_i: S^7\rightarrow \mathrm{SO}(8k-9)$ by
$$\Omega_i(x_1)(x_2,\cdots, x_k)=(x_2,\cdots, x_i, -\overline{x_1}x_{i+1}, x_{i+2},\cdots, x_{k}).$$
It follows that as homotopy classes in $\pi_{8k-2}S^{8k-9}$
$$[H(\Omega]=[H(\Omega_0)]+[H(\Omega_1)]+\cdots+[H(\Omega_{k-2})],$$
$$[H(\Omega_i)]=-[\Sigma^{8k-17} H(\sigma)],~~ i= 1, \cdots, k-2,$$
and
$$[H(\Omega_0)]=-[\Sigma^{8k-16} H(\rho)].$$
Therefore, for $k\geq 3$,
\begin{eqnarray*}
[g_k]&=&[\Phi]=[H(\Omega)]=-(k-2)[\Sigma^{8k-17}H(\sigma)]-[\Sigma^{8k-16} H(\rho)]=-k[\Sigma^{8k-17}H(\sigma)]\\
&\equiv&-k~(\mathrm{mod}~240).
\end{eqnarray*}
Consequently, the sphere bundle $S^{8k-9}\hookrightarrow M_+^{16k-10}\rightarrow V_1(\mathbb{O}^k)=S^{8k-1}$ admits a cross-section if and only if $k$ can be divided by $240$, which implies the definite case of Proposition \ref{octonion}. As we mentioned in Remark \ref{1.4}, the focal submanifold $M_+$ of OT-FKM type in the definite case is diffeomorphic to
$V_2(\mathbb{O}^k)$.

For the sake of comparison, we would like to mention that the symplectic Stiefel fibration $V_4(\mathbb{H}^{2k})\rightarrow V_1(\mathbb{H}^{2k})=S^{8k-1}$ admits a cross-section if and only if $2k\equiv 0 (\text{mod}~c_4)$, where $c_4=2^7\cdot 3^4\cdot 5\cdot 7$ is the quaternion James number (cf. \cite{Ja58}, \cite{SS73}). Notice that when $V_{2}(\mathbb{O}^k)$ has a cross-section, $V_4(\mathbb{H}^{2k})$ might not have.

Moreover, there is a natural embedding $V_p(\mathbb{H}^{m})\subset V_{2p}(\mathbb{C}^{2m})$, and the complex Stiefel fibration $V_8(\mathbb{C}^{4k})\rightarrow V_1(\mathbb{C}^{4k})=S^{8k-1}$ admits a cross-section if and only if $4k\equiv 0 (\text{mod}~M_8)$, where $M_8=b_8$ is the complex James number and $b_8=c_4$ (cf. \cite{AW65}, \cite{SS73}).
It follows that if $V_4(\mathbb{H}^{2k})$ admits a cross-section, then so does $V_8(\mathbb{C}^{4k})$.}
\end{rem}

\subsubsection{}\hspace{-2.5mm}\textbf{The indefinite case.}\label{indef 8}
For $z=(z_1,\cdots, z_k), w=(w_1,\cdots, w_k)\in \mathbb{O}^{k}$ and $0\leq p\leq k-1$, define
\begin{equation}\label{o1}
\langle z, w\rangle_p:=\sum_{i=1}^{p}z_i\overline{w}_i+\sum_{j=p+1}^{k-1}w_j\overline{z}_j+z_{k}\overline{w}_k.
\end{equation}

\begin{prop}
\begin{itemize}
\item[(1).] For $z=(z_1,\cdots, z_k), w=(w_1,\cdots, w_k)\in \mathbb{O}^{k}$, $\langle z, w\rangle_p=0$
if and only if the following two conditions hold:
$\mathrm{Re}\langle z, w\rangle_{\mathbb{O}}=0$ and $\sum_{i=1}^{p}z_i\overline{w}_i-\sum_{j=p+1}^{k-1}z_j\overline{w}_j+z_{k}\overline{w}_k\in \mathbb{R}$.\vspace{2mm}

\item[(2).] $N'_+=\{(z, w)\in \mathbb{O}^{k}\oplus\mathbb{O}^{k}~|~|z|=|w|=1, \langle z, w\rangle_p=0\}$
is diffeomorphic to the focal submanifold $M_+$ of OT-FKM isoparametric hypersurface with $(m_1, m_2)=(8, 8k-9)$ and $\mathrm{Tr}(P_0P_1\cdots P_8)=-16(2p-k+2)$.
\end{itemize}
\end{prop}

\begin{proof}
It is clear that (1) is valid. For (2), define seven orthogonal transformations $E_1, \cdots, E_7$ on $\mathbb{O}^{k}$ as follows.
Given $z=(z_1,\cdots, z_k)\in\mathbb{O}^{k}$,
\begin{equation*}\label{indef E}
E_{\alpha}(z)=(e_{\alpha}z_1, \cdots, e_{\alpha}z_p, -e_{\alpha}z_{p+1},\cdots, -e_{\alpha}z_{k-1}, e_{\alpha}z_{k}), ~~\alpha=1,\cdots,7.
\end{equation*}
Similar to the arguments in the proof of Proposition \ref{1.2},
we see that the following linear transformations $P_0,\ldots, P_8$ on $\R^{16k}$
\begin{equation}\label{indef P8}
P_0(z, w)=(z, -w), ~~P_1(z, w)=(w, z),~~P_{1+\alpha}(z, w)=(E_{\alpha}w, -E_{\alpha}z),
\end{equation}
form a symmetric Clifford system on $\R^{16k}$.
Moreover, a direct computation leads to
\begin{eqnarray*}
P_0P_1\cdots P_8(z,w)&=&(-z_1,\cdots,-z_p, z_{p+1},\cdots, z_{k-1},-z_k,\\
&&-w_1,\cdots, -w_p,w_{p+1},\cdots,w_{k-1}, -w_k),
\end{eqnarray*}
which implies $\mathrm{Tr}(P_0P_1\cdots P_8)=-16(2p-k+2)$. By definition and (1),  the focal submanifold $M_+$ associated
with the given Clifford system is determined by
$M_+=\{(z, w)\in \mathbb{O}^{k}\oplus\mathbb{O}^{k}~|~|z|=|w|=\frac{1}{\sqrt{2}}, \langle z, w\rangle_p=0\}$.
\end{proof}

For each $\varepsilon\in\mathbb{O}$, $W=(w_1,\cdots, w_{k-1})\in \mathbb{O}^{k-1}$, and $0\leq p\leq k-1$, define
\begin{equation}\label{o2}
\varepsilon\ast_p W:=(\varepsilon w_1,\cdots, \varepsilon w_p, \overline{\varepsilon} w_{p+1},\cdots, \overline{\varepsilon} w_{k-1}).
\end{equation}
Moreover, for $X=(x_1,\cdots, x_{k-1}),W \in \mathbb{O}^{k-1}\cong$ $\mathbb{O}^{k-1}\oplus 0\subset \mathbb{O}^{k}$,  recall
$$\langle X, W\rangle_p=\sum_{i=1}^{p}x_i\overline{w}_i+\sum_{j=p+1}^{k-1}w_j\overline{x}_j.$$
Now, we are ready to give a proof to
\vspace{2mm}

\noindent \textbf{Theorem \ref{characteristic maps}.} (ii).
\emph{For $m=8$ and $\mathrm{Tr}(P_0P_1\cdots P_8)=-16(2p-k+2)$ with $0\leq p\leq k-1$, the characteristic map $\chi: S^{8k-2} \rightarrow \mathrm{SO}(8k-8)$ of $\eta$ is given by
$$\chi(z)(X)=X-2\Big(\langle X, W\rangle_p(1+z_k)^{-2}\Big)\ast_p W,$$
where $z=(z_1,..., z_k)\in S^{8k-2}$, the equator in $S^{8k-1}\subset \mathbb{O}^k$ with $ \mathrm{Re}(z_k)=0$, $W=(z_1,..., z_{k-1})$, and $X\in \mathbb{O}^{k-1}\cong \mathbb{R}^{8k-8}.$ Moreover, $\langle \cdot, \cdot\rangle_p$ and $\ast_p$-operation are defined by (\ref{o1}) and (\ref{o2}), respectively.}
\vspace{2mm}

\begin{proof}
As before, for $z=(z_1,\cdots, z_k), w=(w_1,\cdots, w_k)\in \mathbb{O}^{k}$, choose the octonionic inner product in (\ref{H}). Let $S^{8k-1}$ be the unit sphere in $\mathbb{O}^{k}$ and $S^{8k-2}$ be its equator defined by $\mathrm{Re}(z_k)=0$. Let $N=(0,\cdots,0, 1)\in S^{8k-1}$.

Now, we construct an embedding
$$\psi_1: (S^{8k-1}\setminus \{-N\})\times S^{8k-9}\rightarrow N_+, \quad \psi_1(z, X)=(z, Y),$$
where $$X\in \mathbb{O}^{k-1}\cong\mathbb{O}^{k-1}\oplus ~0\subset \mathbb{O}^{k}, W=(z_1,\cdots, z_{k-1}),$$
$$b=(1+z_k)(1+\overline{z_k})^{-1}, \alpha=(1+\overline{z_k})^{-1},$$
and
$$Y=\Big(X-(\langle X, W\rangle_p\alpha)\ast_p W, ~~-\langle X, W\rangle_p b\Big).$$
To prove that $\psi_1$ is well-defined, we need the following lemma.

\begin{lem}\label{2.2}
\begin{itemize}
\item[(1).] For any $\varepsilon\in\mathbb{O}$, $W=(z_1,\cdots, z_{k-1})\in \mathbb{O}^{k-1}$,
$$\langle \varepsilon\ast_p W, W\rangle_p=\varepsilon\langle W, W\rangle_{\mathbb{O}}.$$

\item[(2).] For any $X=(x_1,\cdots, x_{k-1}), W=(z_1,\cdots, z_{k-1})\in \mathbb{O}^{k-1}$ and $z_k\in\mathbb{O}$,
$$\langle X, \varepsilon\ast_p W\rangle=\mathrm{Re}(\alpha)|\langle X, W\rangle_p|^2,$$
where $\alpha=(1+\overline{z_k})^{-1}$ and $\varepsilon=\langle X, W\rangle_p\alpha$.
\end{itemize}
\end{lem}

\begin{proof}
For (1),
\begin{eqnarray*}
\langle \varepsilon\ast_p W, W\rangle_p
&=&(\varepsilon z_1)\overline{z}_1+\cdots (\varepsilon z_p)\overline{z}_p+z_{p+1}(\overline{\bar{\varepsilon} z_{p+1}})+\cdots z_{k-1}(\overline{\bar{\varepsilon} z_{k-1}})\\
&=&(\varepsilon z_1)\overline{z}_1+\cdots (\varepsilon z_p)\overline{z}_p+z_{p+1}(\overline{ z}_{p+1}\varepsilon)+\cdots z_{k-1}(\overline{z}_{k-1}\varepsilon )\\
&=&\varepsilon (z_1\overline{z}_1+\cdots  z_p\overline{z}_p+z_{p+1}\overline{ z}_{p+1}+\cdots z_{k-1}\overline{z}_{k-1})\\
&=&\varepsilon\langle W, W\rangle_{\mathbb{O }}.
\end{eqnarray*}
For (2),
\begin{eqnarray*}
\langle X, \varepsilon \ast_p W\rangle&=&\mathrm{Re}\langle (x_1,\cdots, x_{k-1}), (\varepsilon z_1,\cdots, \varepsilon z_p, \overline{\varepsilon} z_{p+1},\cdots, \overline{\varepsilon} z_{k-1})\rangle_{\mathbb{O}}\nonumber\\
&=&\mathrm{Re}(x_1(\overline{\varepsilon z_1})+\cdots+x_{p}(\overline{\varepsilon z_{p}})+(\overline{\varepsilon}z_{p+1})\overline{x}_{p+1}+\cdots+
(\overline{\varepsilon}z_{k-1})\overline{x}_{k-1})\nonumber\\
&=&\mathrm{Re}(x_1(\overline{ z}_1\overline{\varepsilon})+\cdots+x_{p}(\overline{ z}_{p}\overline{\varepsilon})+(\overline{\varepsilon}z_{p+1})\overline{x}_{p+1}+\cdots+
(\overline{\varepsilon}z_{k-1})\overline{x}_{k-1})\nonumber\\
&=&\mathrm{Re}((x_1\overline{ z}_1)\overline{\varepsilon})+\cdots+(x_{p}\overline{ z}_{p})\overline{\varepsilon}+\overline{\varepsilon}(z_{p+1}\overline{x}_{p+1})+\cdots+
\overline{\varepsilon}(z_{k-1}\overline{x}_{k-1}))\nonumber\\
&=&\mathrm{Re}(\overline{\varepsilon}(x_1\overline{ z}_1+\cdots+x_{p}\overline{ z}_{p}+z_{p+1}\overline{x}_{p+1}+\cdots+
z_{k-1}\overline{x}_{k-1}))\nonumber\\
&=&\mathrm{Re}(\overline{\varepsilon}\langle X, W\rangle_p)\nonumber\\
&=&\mathrm{Re}((\overline{\alpha}\overline{\langle X, W\rangle_p})\langle X, W\rangle_p)\nonumber\\
&=&\mathrm{Re}(\alpha)|\langle X, W\rangle_p|^2.
\end{eqnarray*}
During the proof, the following identities are used
$$\mathrm{Re}(ab)=\mathrm{Re}(ba), ~~\mathrm{Re}(a(bc))=\mathrm{Re}((ab)c)=\mathrm{Re}(c(ab)),~~~ \forall~ a, b, c \in \mathbb{O}.$$
\end{proof}

From (1) of Lemma \ref{2.2}, it follows that
\begin{eqnarray*}
\langle Y, z\rangle_p
&=& \langle ~(X-(\langle X, W\rangle_p\alpha)\ast_p W, -\langle X, W\rangle_p b), ~~~(W, z_k)~\rangle_p\\
&=& \langle X, W\rangle_p-\langle~ (\langle X, W\rangle_p\alpha)\ast_p W,~~ W\rangle_p-(\langle X, W\rangle_p b)\overline{z}_k\\
&=& \langle X, W\rangle_p-(\langle X, W\rangle_p\alpha) \langle W, W\rangle_{\mathbb{O}}-(\langle X, W\rangle_p b)\overline{z}_k\\
&=& \langle X, W\rangle_p~(1-|W|^2\alpha-b \overline{z}_k)\\
&=& 0,
\end{eqnarray*}
where we have used a trivial equality $1-|W|^2\alpha-b \overline{z}_k=0$.
From (2) of Lemma \ref{2.2}, it follows that
\begin{eqnarray*}
|Y|^2&=&|X-(\langle X, W\rangle_{p}\alpha)\ast_p W|^2+|\langle X, W\rangle_{p} b|^2\\
&=&|X|^2+(|\alpha|^2|W|^2+|b|^2)\cdot|\langle X, W\rangle_{p}|^2-2\langle X, (\langle X, W\rangle_{p}\alpha)\ast_p W\rangle\\
&=& 1+(|\alpha|^2|W|^2+|b|^2)\cdot|\langle X, W\rangle_{p}|^2-2\mathrm{Re}(\alpha)|\langle X, W\rangle_p|^2
\\
&=& 1+(|\alpha|^2|W|^2+|b|^2-2\mathrm{Re}(\alpha))\cdot|\langle X, W\rangle_{p}|^2\\
&=&1,
\end{eqnarray*}
where we have used a trivial equality $|\alpha|^2|W|^2+|b|^2-2\mathrm{Re}(\alpha)=0$.
Hence, $\psi_1$ is well-defined.

Meanwhile, we construct another embedding
$$\psi_2: (S^{8k-1}\setminus \{N\})\times S^{8k-9}\rightarrow N_+, \quad \psi_2(z, X)=(z, Y),$$
where $$X\in \mathbb{O}^{k-1}\cong\mathbb{O}^{k-1}\oplus~ 0\subset \mathbb{O}^{k}, W=(z_1,\cdots, z_{k-1}),$$
$$ b'=(1-z_k)(1-\overline{z_k})^{-1}, \alpha'=(1-\overline{z_k})^{-1},$$
and
$$Y=\Big(X-(\langle X, W\rangle_p\alpha')\ast_p W, ~~\langle X, W\rangle_p b'\Big).$$
Similarly, $\psi_2$ is well-defined.

Thus we get the characteristic map
$$\chi: S^{8k-2}\rightarrow \mathrm{SO}(8k-8), \quad z\mapsto \chi(z),$$
where $\chi(z)$ is defined by
$$\chi(z)(X):=X-2(\langle X, W\rangle_p(1+z_k)^{-2})\ast_p W \quad\text{for}~ X\in \mathbb{O}^{k-1}.$$

Under the standard identification $\mathbb{O}^{k-1}\cong \mathbb{R}^{8k-8}$, for each $z\in S^{8k-2}$, it follows that $\chi(z)$ is
a real linear transformation and preserves the real inner product. In particular, $\chi(0, z_k)=I_{8k-8}$, $\forall ~(0, z_k)\in S^{8k-2}$.
 Hence, for each $z\in S^{8k-2}$, $\chi(z)\in \mathrm{SO}(8k-8)$.
Therefore, $\chi$ is well-defined.
The proof of Theorem \ref{characteristic maps} (ii) is complete.
\end{proof}

\begin{rem}\rm{
Let $k=2$ and $p=0$. This is just the indefinite case with $(m_1, m_2)=(8, 7)$. By \cite{FKM81}, the focal submanifold $M_+$ in this case is inhomogeneously embedded in $S^{31}$. The characteristic map is given by
\begin{eqnarray*}
\chi: S^{14}&\longrightarrow& \mathrm{SO}(8)\\
\chi(z_1, z_2)(X)&=&X-2((1-z_2)^{-2}(X\overline{z}_1))z_1, ~~\forall ~~X\in \mathbb{O},
\end{eqnarray*}
where $\mathrm{Re}(z_2)=0$. Let $X=(1,0,\cdots,0)\in S^7$. Then
$$\chi(z)X=1-2((1-z_2)^{-2}\overline{z}_1)z_1=-\frac{(1+z_2)^2}{(1-z_2)^2}.$$
Composing with the natural projection $\pi: SO(8)\rightarrow S^7$ induced by $X=(1,0,\cdots,0)$ $\in S^7$ will give rise to a map
\begin{eqnarray*}
g_{2, 0}:=\pi\circ\chi:~~~ S^{14}&\longrightarrow & S^7\\
z=(z_1, z_2)&\mapsto& -\frac{(1+z_2)^2}{(1-z_2)^2}.
\end{eqnarray*}
A similar  argument as that in Remark \ref{1.1} reveals that $g_{2, 0}$ is homotopic to a constant map, $g_{2, 0}\simeq c$. Thus the sphere bundle $S^7\hookrightarrow M_+^{22}\rightarrow S^{15}$ admits a cross-section.}
\end{rem}

\begin{rem}\rm{
In this case with $(m_1,m_2)=(8,8k-9)$ and $0\leq p \leq k-1$, recall the characteristic map
is expressed by
$$\chi(z)(X):=X-2(\langle X, W\rangle_p(1+z_k)^{-2})\ast_p W \quad\text{for}~ X\in \mathbb{O}^{k-1}.$$
Set $X=(1,0,\cdots,0)\in S^{8k-9}$. For $W=(z_1,\cdots,z_{k-1})$, observe
\begin{equation*}
\langle X, W\rangle_p=\left\{\begin{array}{ll}
\overline{z_1}, ~~p\geq 1\\
z_1, ~~p=0.
\end{array}\right.
\end{equation*}
Combining with the natural projection $\pi: SO(8k-8)\rightarrow S^{8k-9}$ induced by $X=(1,0,\cdots,0)\in S^{8k-9}$,
we obtain
$$g_{k, p}:=\pi\circ\chi:~~S^{8k-2}\longrightarrow S^{8k-9},$$
which is defined by
\begin{eqnarray}\label{g p}
&&\\
&&g_{k, p}(z)=\small{\left\{\begin{array}{ll}
\Big(1-2\overline{z_1}\frac{1}{(1+z_k)^2}z_1, ~-2(\overline{z_1}\frac{1}{(1+z_k)^2})z_ 2,\cdots,~ -2(\overline{z_1}\frac{1}{(1+z_k)^2})z_ p,~~\\
\hspace{2.7cm}-2(\frac{1}{(1-z_k)^2}z_1)z_{p+1},  \cdots, -2(\frac{1}{(1-z_k)^2}z_1)z_{k-1}\Big),~~p\geq 1,\vspace{3mm}\\
\Big(1-2\frac{1}{(1-z_k)^2}|z_1|^2, -2(\frac{1}{(1-z_k)^2}\overline{z_1})z_2,\cdots,-2(\frac{1}{(1-z_k)^2}\overline{z_1})z_{k-1}\Big),~~p=0,
\end{array}\right.}\nonumber
\end{eqnarray}
where $z=(W, z_k)$ and $\mathrm{Re}(z_k)=0$.}
\end{rem}

For the definite case with $p=k-1$, we have dealt with the homotopy class $[g_{k, p}]=[g_k]$ in Remark \ref{1.8}. From now on, we focus on the indefinite case, i.e., the case with $p\neq k-1$. Actually, we can apply the similar ideas in the proof of Proposition \ref{quaternion}, and get similar conclusions, since we need not to use the associative law. The following algebraic fact will be used, for which the proof is similar to that of Lemma \ref{extend} and is omitted.
\begin{lem}\label{extend8}
When $m=8$, the symmetric Clifford system $\{P_0,\cdots,P_8\}$ of indefinite type on $\mathbb{R}^{16k}$ defined in (\ref{indef P8}) can be extended to a symmetric Clifford system $\{P_0,\cdots, P_9\}$ if and only if $k=2p+2$.
\end{lem}
Now we are in a position to prove Proposition \ref{octonion}. During the proof, the homotopy class $[g_{k, p}]$ will also be determined.
\vspace{2mm}

\noindent \textbf{Proposition \ref{octonion}.} \emph{Let $m=8$ and $\mathrm{Tr}(P_0P_1\cdots P_8)=-16(2p-k+2)$ with $0\leq p\leq k-1$. Then the sphere bundle $S^{8k-9}\hookrightarrow M_+^{16k-10}\rightarrow S^{8k-1}$ associated with the vector bundle $\eta$ admits a cross-section if and only if $k-2-2p$ can be divided by $240$.}

\emph{
In particular, the sphere bundle associated with the octonionic Stiefel manifold $V_2(\mathbb{O}^k)$: $S^{8k-9}\hookrightarrow V_2(\mathbb{O}^k)\rightarrow V_1(\mathbb{O}^k)$ admits a cross-section if and only if $k$ can be divided by $240$.}
\vspace{2mm}

\begin{proof}
As an immediate consequence of Lemma \ref{extend8}, in the case with $k=2p+2$, the sphere bundle $S^{8k-9}\hookrightarrow M_+^{16k-10}\rightarrow S^{8k-1}$ admits a cross-section, thus $[g_{k, p}]=0$ in $\pi_{8k-2}S^{8k-9}\cong Z_{240}$.

For $p=0$ and general $k$, recall $g_{k, 0}: S^{8k-2}\rightarrow S^{8k-9}$ in (\ref{g p})
\begin{eqnarray*}
g_{k, 0}(z_1, z_2,\cdots, z_k)&=&\Big(1-2(1-z_k)^{-2}|z_1|^2, -2((1-z_k)^{-2}\overline{z_1})z_2,\cdots,\\
&&\hspace{4cm}-2((1-z_k)^{-2}\overline{z_1})z_{k-1}\Big).
\end{eqnarray*}
Define a similar continuous function $f: S^7\times S^{8k-10}\rightarrow S^{8k-10}$ as that in Remark \ref{indef 4 rem}, and consider the corresponding Hopf construction $H(f): S^{8k-2}\rightarrow S^{8k-9}$. Using proof by contradiction, one can see that $(g_{k, 0}+H(f))(z)\neq 0$ for any $z\in S^{8k-1}$ with $\mathrm{Re}(z_k)=0$.
Therefore, $g_{k, 0}\simeq H(f)$, and thus $$[g_{k, 0}]=[H(f)]\equiv -(k-2) ~(\mathrm{mod}~ 240 )\quad\text{in} ~~\pi_{8k-2}S^{8k-9}\cong Z_{240}.$$

For $1\leq p< k-1$, recall $g_{k, p}: S^{8k-2}\rightarrow S^{8k-9}$ in (\ref{g p}) \begin{eqnarray*}
g_{k, p}(z)&=&
\Big(1-2\overline{z_1}\frac{1}{(1+z_k)^2}z_1, ~-2(\overline{z_1}\frac{1}{(1+z_k)^2})z_ 2,\cdots,~ -2(\overline{z_1}\frac{1}{(1+z_k)^2})z_ p\\
&&\hspace{2.5cm}-2(\frac{1}{(1-z_k)^2}z_1)z_{p+1},  \cdots, -2(\frac{1}{(1-z_k)^2}z_1)z_{k-1}\Big).
\end{eqnarray*}
Define $\Upsilon:~S^{8k-2}\longrightarrow S^{8k-9}$ by
\begin{eqnarray*}
\Upsilon(z)&=&\Big(1-2\overline{z_1}(1+z_k)^{-2}z_1, -\frac{2\overline{z_1}z_2}{1+|z_k|^2}, \cdots, -\frac{2\overline{z_1}z_p}{1+|z_k|^2}, -\frac{2z_1z_{p+1}}{1+|z_k|^2}, \cdots, \\
&&\hspace{9cm}-\frac{2z_1z_{k-1}}{1+|z_k|^2}\Big),
\end{eqnarray*}
and $\Psi:~S^{8k-2}\longrightarrow S^{8k-9}$ by
$$\Psi(z)=\Big(1-2|z_1|^2+2\frac{\overline{z_1}z_kz_1}{|z_1|}, -2\overline{z_1}z_2,\cdots, -2\overline{z_1}z_p, -2z_1z_{p+1}, \cdots, -2z_1z_{k-1}\Big).$$
Again, using proof by contradiction, one can show that $(g_{k,p}+\Upsilon)(z)\neq 0$ and $(\Upsilon+\Psi)(z)\neq 0$ for any $z\in S^{8k-1}$ with $\mathrm{Re}(z_k)=0$. Thus
$\Psi\simeq\Upsilon\simeq g_{k, p}$.

Consequently, by virtue of the discussion in Remark \ref{1.8}, we obtain
\begin{eqnarray*}
[g_{k, p}]&=&[\Upsilon]=[\Psi]\\
&=&[H(\Omega_0)]+[H(\Omega_1)]+\cdots+[H(\Omega_p)]+(k-1-p)[\Sigma^{4k-9}H(\sigma)]\\
&=&-[\Sigma^{8k-16}H(\rho)]-(p-1)[\Sigma^{8k-17}H(\sigma)]+(k-1-p)[\Sigma^{8k-17}H(\sigma)]\\
&\equiv& k-2-2p~(\mathrm{mod}~240).
\end{eqnarray*}

Therefore, the sphere bundle $S^{8k-9}\hookrightarrow M_+^{16k-10}\rightarrow S^{8k-1}$ admits a cross-section if and only if $k-2-2p$ can be divided by $240$.
\end{proof}

\begin{rem}\rm{
Using the product $\langle \cdot, \cdot\rangle_p$ of $\mathbb{O}^k$ defined by (\ref{o1}), one can also define a space $N_{p,k,q}(\mathbb{O})$ similarly as $N_{p,k,q}(\mathbb{H})$ in subsection \ref{Natural questions} and ask the same questions. In \cite{QTY22}, the authors proved that $N_{p,k,3}(\mathbb{O})$ admits a smooth manifold structure. In particular, for $p=k-1$, $V_3(\mathbb{O}^k)=N_{k-1,k,3}(\mathbb{O})$ is a smooth manifold. This answers partially a question of James \cite{Ja58}.}
\end{rem}

\section{Harmonic maps}
As we mentioned in Section \ref{intro}, given a map $f:M\rightarrow N$ between Riemannian manifolds, it is natural to seek whether a harmonic representative in the homotopy class of $f$ exists. When the target has non-positive sectional curvatures, Eells and Sampson \cite{ES64} obtained an affirmative answer
by using the method of the heat equation.
In the case of positive sectional curvature, the problem is difficult and challenging.
In the particular case of spheres, Smith \cite{Smi75} proved that $\pi_n S^n$ is represented by harmonic maps for $1\leq n\leq 7$. Moreover, it is an intriguing problem asked by Yau \cite{Yau82} for each elements in $\pi_i S^n$. In \cite{PT97, PT98}, Peng and Tang applied orthogonal multiplication to prove that for almost every odd integer $n$, the element $2$ in $\pi_nS^n=\mathbb{Z}$ admits a harmonic representative. For the case of $i>n$, certain elements in $\pi_3S^2$,
$\pi_7S^n(3\leq n\leq 5)$ and $\pi_{n+3}S^n(5\leq n\leq 10)$ have been considered in
\cite{Smi75}. Furthermore, Peng and Tang \cite{PT96, PT98} have proved that certain elements in $\pi_6S^3$, $\pi_{n+1}S^n(n=10, 12 ,16)$, $\pi_{8k+7}S^{8k+6}(k\geq 0)$ and $\pi_{4k-1}S^{4k-4}(k>2)$ admit harmonic representatives. In this section, based on our study on characteristic maps, we will prove the following:

\vspace{2mm}
\noindent \textbf{Theorem \ref{thm1.2}.}
(i). \emph{The homotopy group $\pi_{14}S^7$ has a harmonic generator.}

(ii). \emph{Let $k>2$. Then the element $(2j-k+1) (\mathrm{mod}~240)$ in the homotopy group $\pi_{8k-1}S^{8k-8}=\mathbb{Z}_{240}$ has a harmonic  representative for every $j=0,\cdots, k-1$.}
\vspace{2mm}

\begin{proof}
For (i), define $f: S^7\times S^6 \rightarrow S^6$ by $$f(x, y)=xy\overline{x}$$
for $x\in \mathbb{O}$, $|x|=1$, $y\in \mathrm{Im} \mathbb{O}$ and $|y|=1$.
Clearly $f$ is a bi-eigenmap with bi-eigenvalue $(24, 6)$. Moreover, its Hopf construction $H(f): S^{14}\rightarrow S^7$ is represented by
$$H(f)(z_1, z_2)=2|z_2|^2-1+2\frac{z_1z_2\overline{z_1}}{|z_1|}$$
for $z_1\in \mathbb{H}$ and $z_2\in \mathrm{Im} \mathbb{H}$. It follows from the Ratto-Ding theorem (\cite{Di94}) that the Hopf construction $H(f)$ is homotopic to a harmonic map.

We are left to show $[H(f)]$ generates the homotopy group $\pi_{14}S^7=\mathbb{Z}_{120}$.
In fact, following the detailed arguments in Remark \ref{1.8}, we can see
 $$[H(f)]=-[\tau]=-[g_2]=-[b']=-[b]\in \pi_{14}S^7.$$
 Moreover according to
\cite{DMR04}, $[b]$ generates $\pi_{14}S^7=\mathbb{Z}_{120}$. Consequently, $[H(f)]$ is also a generator of
the homotopy group $\pi_{14}S^7$.

For (ii), we first consider the map $f: S^7\times S^{8k-9}\rightarrow S^{8k-9}$ defined by
$$f(x_1, x_2,\cdots, x_k)=(x_1x_2,\cdots, x_1x_{j+1}, \overline{x_1}x_{j+2},\cdots, \overline{x_1}x_{k}),$$
where $x_1\in S^7\subset\mathbb{O}$, $(x_2,\cdots, x_{k})\in S^{8k-9}\subset\mathbb{O}^{k-1}$.
It follows from the Ratto-Ding theorem (\cite{Di94}) that the Hopf construction $H(f)$ is homotopic to a harmonic map.

We are now left to determine the homotopy class $[H(f)]\in \pi_{8k-1}S^{8k-8}=\mathbb{Z}_{240}$.
Similar to Remark \ref{1.8}, for $1\leq i\leq j$, define $\eta_i: S^7 \rightarrow \mathrm{SO}(8k-8)$ by
$$\eta_i(x_1)(x_2,\cdots, x_k)=(x_2,\cdots, x_i, x_1x_{i+1},\cdots, x_{k}),$$
and for $j+1\leq i\leq k-1$, define $\eta_i: S^7\rightarrow \mathrm{SO}(8k-8)$ by
$$\eta_i(x_1)(x_2,\cdots, x_k)=(x_2,\cdots, x_i, \overline{x_1}x_{i+1},\cdots, x_{k}).$$
As homotopy classes in $\pi_{8k-1}S^{8k-8}$, the following equalities hold
$$[H(f)]=[H(\eta_1)]+[H(\eta_2)]+\cdots+[H(\eta_{k-1})],$$
$$[H(\eta_i)]=[\Sigma^{8k-16} H(\sigma)],~~ i=1,\cdots, j,$$
and
$$[H(\eta_i)]=-[\Sigma^{8k-16} H(\sigma)],~~ i= j+1,\cdots, k-1,$$
where $\sigma$ and $\Sigma$ are defined in Remark \ref{1.8}.
Therefore, for $k>2$,
\begin{eqnarray*}
[H(f)]&=&j[\Sigma^{8k-16} H(\sigma)]-(k-j-1)[\Sigma^{8k-16} H(\sigma)]\\
&\equiv&2j-k+1~(\mathrm{mod}~240)
\end{eqnarray*}
as we hoped.
\end{proof}


\section{Non-negative curvature}\label{sec4}
\subsection{$M_-$ of OT-FKM type}
This subsection is based on \cite{ABS64} and \cite{Wa88}. Let
$\{P_0, P_1,\cdots, P_{m}\}$ on $\mathbb{R}^{2l}$ be a given symmetric Clifford system. Choose a set of orthogonal transformations $\{E_1, E_2,\cdots, E_{m-1}\}$ on $\mathbb{R}^l$ with respect to the Euclidean metric, which satisfies
$E_{\alpha}E_{\beta}+E_{\beta}E_{\alpha}=-2\delta_{\alpha\beta}I_l$\,\,
for $1 \leq \alpha,\beta\leq m-1$ and up to algebraic equivalence
\begin{equation*}\label{FKM}
P_0=\left(
\begin{matrix}
I_l & 0 \\
0 & -I_l \\
\end{matrix}\right),\;
P_1=\left(
\begin{matrix}
0 & I_l\\
I_l & 0 \\
\end{matrix}\right), P_{1+\alpha}=\left(
\begin{matrix}
0 & E_{\alpha}\\
-E_{\alpha} & 0 \\
\end{matrix}\right),
~\mathrm{for}~ 1 \leq \alpha\leq m-1.
\end{equation*}
Let $\mathfrak{h}(\mathbb{R}^{2l})$ be the space of symmetric endomorphisms on $\mathbb{R}^{2l}$, $\langle\cdot,\cdot \rangle$ be the scalar product on $\mathfrak{h}(\mathbb{R}^{2l})$ defined by
$\langle A, B\rangle:=\frac{1}{2l}\mathrm{Trace}(AB)$, and $E_{\pm}(A)$ be the eienspaces of $A$ with eigenvalues $\pm$ respectively. Denote the unit sphere in $\mathrm{Span}\{P_0,..., P_m\}\subset \mathfrak{h}(\mathbb{R}^{2l})$ by $\Sigma(P_0,..., P_m)$, the so-called
\emph{Clifford sphere}.
According to \cite{FKM81}, the associated focal submanifold $M_-$ of the OT-FKM type is equal to
$$\{x\in S^{2l-1}~|~\exists P\in\Sigma(P_0,..., P_m)~ \text{such that}~ x\in E_+(P) \}.$$
Clearly, it is diffeomorphic to $S(\xi)$,
the total space of an $S^{l-1}$-bundle over the Clifford sphere $\Sigma(P_0,..., P_m)$.

On the other hand, the set of orthogonal transformations $\{E_1, E_2,\cdots, E_{m-1}\}$ on $\mathbb{R}^l$ induces an orthogonal representation of the Clifford algebra $\mathcal{C}_{m-1}$ on $\mathbb{R}^l$. Recall that $\mathrm{Spin}(m)\subset \mathcal{C}_{m}^0\cong \mathcal{C}_{m-1}$, it further induces an orthogonal representation of $\mathrm{Spin}(m)$. For the canonical principal bundle $\mathrm{Spin}(m)\hookrightarrow \mathrm{Spin}(m+1)\rightarrow S^m$ and the representation of $\mathrm{Spin}(m)$ on $\mathbb{R}^l$, we can construct the associated vector bundle $\xi'$ with the total space $\mathrm{Spin}(m+1)\times_{\mathrm{Spin}(m)}\mathbb{R}^l$. Proposition 1 in \cite{Wa88} implies clearly the following
\begin{assert}(\cite{Wa88})\label{4.1.1}
As vector bundles, $\xi$ and $\xi'$ are isomorphic.
\end{assert}
With these preparations, we can show
\vspace{2mm}

\noindent\textbf{Proposition \ref{M_-}.}
\emph{Each focal submanifold $M_-$ of OT-FKM type admits a metric with non-negative sectional curvature.}
\vspace{2mm}

\begin{proof}
By Assertion \ref{4.1.1}, $M_-$, the total space of sphere bundle of $\xi$, is diffeomorphic to that of the sphere bundle of $\xi'$. It means that $M_-$ is diffeomorphic to $\mathrm{Spin}(m+1)\times_{\mathrm{Spin}(m)}S^{l-1}$.

Choose a bi-invariant metric on $\mathrm{Spin}(m+1)$, the standard metric on $S^{l-1}$, and
the product metric on $\mathrm{Spin}(m+1)\times S^{l-1}$. It follows that $\mathrm{Spin}(m)$ acts isometrically
on $\mathrm{Spin}(m+1)\times S^{l-1}$. Therefore, it induces a metric on $M_-\cong \mathrm{Spin}(m+1)\times_{\mathrm{Spin}(m)}S^{l-1}$
with non-negative sectional curvature by the celebrated Gray-O'Neill formula for
Riemannian submersions.
\end{proof}

\subsection{$M_+$ of OT-FKM type and $(m_1, m_2)=(3, 4k-4)$}
Let $m=3$, $l=k\delta(3)=4k$.
For $x=(x_1,\cdots, x_k), y=(y_1,\cdots, y_k)\in \mathbb{H}^k\cong \mathbb{R}^{4k}$, define
$$\langle x, y\rangle_{\mathbb{H}}=\sum_{i=1}^{k}x_i\overline{y}_i,$$
so that $\langle x, y\rangle= \mathrm{Re}\langle x, y\rangle_{\mathbb{H}}.$

Choose a symmetric Clifford system $\{P_0, P_1, P_2, P_3\}$ on $\mathbb{R}^{8k}\cong \mathbb{H}^k\oplus \mathbb{H}^k$ as follows.
For each $(x, y)\in \mathbb{H}^k\oplus \mathbb{H}^k$, define
\begin{eqnarray*}
P_0(x, y)&=&(x, -y),\nonumber\\
P_1(x, y)&=&(\mathrm{i}y, -\mathrm{i}x),\nonumber\\
P_2(x, y)&=&(\mathrm{j}y, -\mathrm{j}x),\nonumber\\
P_3(x, y)&=&(\mathrm{k}y, -\mathrm{k}x).
\end{eqnarray*}
The isoparametric foliation of OT-FKM type in $S^{8k-1}$ with respect to this given Clifford system is inhomogeneous, which was originally discovered by \cite{OT75}. Moreover, the focal submanifold $M_-$ is homogeneously embedded, and $M_+$ are inhomogeneously embedded. Clearly, the focal submanifold $M_+$ associated with the given Clifford system is
$$M_+=\{(x, y)\in \mathbb{H}^k\oplus \mathbb{H}^k~|~\langle x, x\rangle=\langle y, y\rangle=\frac{1}{2},~\langle x, y\rangle_{\mathbb{H}}\in \mathbb{R}\}.$$
As usual let $\mathrm{Sp}(k)$ be the compact symplectic group. Consider the
diagonal action of $\mathrm{Sp}(k)$ on $\mathbb{H}^k\oplus \mathbb{H}^k$. More precisely, for $A\in \mathrm{Sp}(k)$ and $(x, y)\in \mathbb{H}^k\oplus \mathbb{H}^k$, define $A\cdot(x, y)=(xA^{-1}, yA^{-1})$.
\begin{prop}
There exists an induced $\mathrm{Sp}(k)$ action on $M_+$. Moreover, $(M_+, \mathrm{Sp}(k))$ is a cohomogeneity one manifold.
\end{prop}
\begin{proof}
For $x, y\in \mathbb{H}^k$ and $A\in \mathrm{Sp}(k)$, it is easy to verify
\begin{eqnarray*}
|xA|^2=|x|^2=\frac{1}{2},&&|yA|^2=|y|^2=\frac{1}{2},\nonumber\\
\langle xA, yA\rangle_{\mathbb{H}}&=&\langle x, y\rangle_{\mathbb{H}} \in \mathbb{R}.
\end{eqnarray*}
Therefore, $M_+$ is an invariant subset and has an induced $\mathrm{Sp}(k)$ action.
Define a real function $f: M_+\rightarrow \mathbb{R}$ by $f(x, y)=2\langle x, y\rangle$. According to
\cite{QT16}, $f$ is an isoparametric function with $\mathrm{Im}(f)=[-1, 1]$. For
$t\in (-1, 1)$, $f^{-1}(t)$ is diffeomorphic to the focal submanifold of the definite $(4, 4k-5)$
case with dimension equal to $8k-6$, and $f^{-1}(0)$ is the given focal submanifold, which
is homogeneously embedded in $S^{8k-1}$. For each $t\in \{1, -1\}$, $f^{-1}(t)$ is isometric to $S^{4k-1}$
and totally geodesic in $M_+$. Meanwhile, for any $(x_0, y_0)\in M_+$ with $f(x_0, y_0)=t
\in (-1, 1)$, the orbit $\mathrm{Sp}(k)\cdot (x_0, y_0)$ is exactly $f^{-1}(t)$. This means that
$(M_+, \mathrm{Sp}(k))$ is a cohomogeneity one manifold.
\end{proof}
Let $S^{4k-4}$ be the unit sphere in $\mathbb{R}\oplus \mathbb{H}^{k-1}\subset \mathbb{H}\oplus \mathbb{H}^{k-1}$
and $\mathrm{Sp}(k-1)$ be identified with the subgroup of $\mathrm{Sp}(k)$
$$\left\{\left(
\begin{array}{cc}
1 & 0 \\
0 & g \\
\end{array}\right)\;\in \mathrm{Sp}(k)~|~g\in \mathrm{Sp}(k-1)\right\}.$$
Let the symplectic group $\mathrm{Sp}(k-1)$ act on $\mathrm{Sp}(k)\times S^{4k-4}$ as
\begin{eqnarray*}
\mathrm{Sp}(k-1)\times (\mathrm{Sp}(k)\times S^{4k-4})&\rightarrow&\mathrm{Sp}(k)\times S^{4k-4},\nonumber\\
(h, (A, z))&\mapsto&(hA, zh^{-1}).
\end{eqnarray*}
It is clear that the action is free. Denote the quotient space by $\mathrm{Sp}(k)\times_{\mathrm{Sp}(k-1)}S^{4k-4}$.
\begin{prop}
With the notations as above, $M_+$ is diffeomorphic to
$$\mathrm{Sp}(k)\times_{\mathrm{Sp}(k-1)}S^{4k-4}.$$
\end{prop}
\begin{proof}
Define
\begin{eqnarray*}
\Phi: \mathrm{Sp}(k)\times S^{4k-4}&\rightarrow&M_+,\nonumber\\
(A, z)&\mapsto&\frac{1}{\sqrt{2}}(A_1, zA),
\end{eqnarray*}
where $A_1$ is the first row of $A$, and $z=(z_1, z_2,\cdots, z_k)\in \mathbb{R}\oplus \mathbb{H}^{k-1}\subset \mathbb{H}\oplus \mathbb{H}^{k-1}$.
Since $\langle A_1, zA\rangle_{\mathbb{H}}=\langle A_1\overline{A}^t, z\rangle_{\mathbb{H}}=z_1\in \mathbb{R}$,
$\Phi$ is well-defined. Define an $\mathrm{Sp}(k)$ action on $\mathrm{Sp}(k)\times S^{4k-4}$ by
$$g\ast(A, z)=(Ag^{-1}, z),~\mathrm{for}~g\in \mathrm{Sp}(k)~\mathrm{and}~(A, z)\in \mathrm{Sp}(k)\times S^{4k-4}.$$
Clearly, this $\mathrm{Sp}(k)$ action commutes with the $\mathrm{Sp}(k-1)$ action on $\mathrm{Sp}(k)\times S^{4k-4}$ defined before. Thus $\Phi$ is an $\mathrm{Sp}(k)$-equivariant map, and a surjective submersion whose fibers are exactly the orbits of the $\mathrm{Sp}(k-1)$ action.
As a result, $\Phi$ induces the desired $\mathrm{Sp}(k)$-equivariant diffeomorphism.
\end{proof}
We conclude this section with a proof of
\vspace{2mm}

\noindent \textbf{Theorem \ref{M_+}.}
\emph{The focal submanifold $M_+$ of OT-FKM type with $(m_1, m_2)=(3, 4k-4)$ admits a metric with non-negative
sectional curvature.}
\vspace{2mm}

\begin{proof}
Choose a bi-invariant metric on $\mathrm{Sp}(k)$, the standard metric on $S^{4k-4}$, and
the product metric on $\mathrm{Sp}(k)\times S^{4k-4}$. It follows that $\mathrm{Sp}(k-1)$
acts isometrically on $\mathrm{Sp}(k)\times S^{4k-4}$. By the previous proposition, it induces a metric
on $M_+\cong\mathrm{Sp}(k)\times_{\mathrm{Sp}(k-1)}S^{4k-4}$ with non-negative sectional curvature
by the celebrated Gray-O'Neill formula for Riemannian submersions.
\end{proof}

\begin{ack}
The authors express their gratitude to the anonymous referees for careful reading and valuable suggestions.
\end{ack}

\end{document}